\documentclass[onefignum,onetabnum]{siamonline250211}

\usepackage{amsmath,amssymb,mathtools}
\usepackage{bm}
\usepackage{algorithm,algorithmic}
\usepackage{graphicx}
\usepackage{booktabs}
\usepackage{enumitem}
\usepackage{multirow}
\usepackage{url}

\usepackage{xr}
\newsiamthm{assumption}{Assumption}
\newsiamremark{remark}{Remark}
\newsiamremark{example}{Example}

\newcommand{\R}{\mathbb{R}}
\newcommand{\E}{\mathbb{E}}
\newcommand{\SO}{\mathrm{SO}}
\newcommand{\SU}{\mathrm{SU}}
\newcommand{\SE}{\mathrm{SE}}
\newcommand{\so}{\mathfrak{so}}
\newcommand{\se}{\mathfrak{se}}
\newcommand{\ssl}{\mathfrak{sl}}
\newcommand{\GL}{\mathrm{GL}}
\newcommand{\gl}{\mathfrak{gl}}

\newcommand{\uu}{\mathfrak{u}}
\newcommand{\su}{\mathfrak{su}}
\newcommand{\tr}{\mathrm{tr}}
\newcommand{\norm}[1]{\left\lVert #1\right\rVert}
\newcommand{\ip}[2]{\left\langle #1,#2\right\rangle}

\providecommand{\g}{}
\renewcommand{\g}{\mathfrak{g}}

\title{
A Representation--Optimization Dichotomy, Lie-Algebraic Policy Optimization%
\thanks{\textit{Preprint. This work is currently under review at the SIAM Journal on Mathematics of Data Science (SIMODS).}}
}

\author{
Sooraj K.C%
\thanks{Department of Pure and Applied Mathematics, Alliance University, Bengaluru, India
(\email{ksoorajPHD23@sam.alliance.edu.in}).}
\and
Vivek Mishra%
\thanks{Department of Pure and Applied Mathematics, Alliance University, Bengaluru, India.}
}

\headers{A Representation--Optimization Dichotomy, Lie-Algebraic Policy Optimization}{S.~K.C\ and V.~Mishra}

\begin{document}
\maketitle

\begin{abstract}
Structured reinforcement learning and stochastic optimization problems 
often involve parameters evolving on matrix Lie groups such as rotations 
and rigid-body transformations. We establish a sharp 
\emph{representation--optimization dichotomy} for 
Lie-algebra--parameterized Gaussian policy objectives in the Lie Group 
MDP class: the gradient Lipschitz constant $L(R)$---which governs 
step-size selection, convergence rates, and sample complexity of any 
first-order method applied to this class---is determined by the algebraic 
type of $\g$, uniformly over all objectives in the Lie Group MDP class, 
independently of the specific reward structure or transition kernel. Specifically: $L = \mathcal{O}(1)$ 
for compact $\g$ (e.g., $\so(n)$, $\su(n)$), and 
$L = \Theta(e^{2R})$ for $\g = \gl(n)$, with $\mathcal{O}(e^{2R})$ 
holding for all algebras with a hyperbolic element. The key analytical step is a lower bound showing that the exponential 
growth cannot be removed by cancellations between the exponential map and 
the objective, establishing the dichotomy as intrinsic to the algebra type.
At its core, this is a structural smoothness theorem for matrix-Lie-parameterized 
stochastic objectives, independent of the specific reinforcement-learning 
interpretation. The dichotomy has a direct algorithmic implication for compact algebras: 
radius-independent smoothness enables $\mathcal{O}(1/\sqrt{T})$ 
convergence with an $\mathcal{O}(n^2 J)$ Lie-algebraic projection step 
in place of $\mathcal{O}(d_{\g}^3)$ Fisher inversion. A Kantorovich 
alignment bound $\alpha \ge 2\sqrt{\kappa}/(\kappa+1)$ provides a 
computable diagnostic for when this Euclidean projection adequately 
surrogates full natural gradient inversion. Controlled experiments on $\SO(3)^J$ and $\SE(3)$ validate the 
theoretical predictions: radius-independent smoothness for compact 
algebras, polynomial Lipschitz growth for $\SE(3)$ (whose algebra 
contains no hyperbolic elements and lies outside the exponential 
regime), and alignment bounds across condition-number regimes. The projection step runs $1.1$--$1.7\times$ 
faster than Cholesky Fisher inversion at benchmarked scales, with the 
$\mathcal{O}(n^2 J)$ vs.\ $\mathcal{O}(d_{\g}^3)$ asymptotic 
advantage growing substantially at larger joint counts.
\end{abstract}



\section{Introduction}
\label{sec:intro}

A recurring challenge in structured data-driven optimization is 
understanding how the geometry of the parameter space shapes the 
computational difficulty of the optimization problem. When model 
parameters lie on a matrix Lie group or its associated Lie algebra---as 
in geometric estimation, rotation-valued signal processing, trajectory 
optimization for mechanical systems, and structured reinforcement 
learning---the relevant algebraic structure is not merely a 
representational choice: it directly governs the growth of the gradient 
Lipschitz constant, the permissible step sizes of first-order methods, 
and the resulting convergence rates. The central question is: 
\emph{which policy parameterizations admit efficient gradient-based optimization, and what structural property of the parameterization determines this?}

This paper establishes a sharp answer within the Lie Group MDP 
class (Definition~\ref{def:lie_mdp}). 
When the underlying control system possesses geometric 
structure---as in robotic manipulation, legged locomotion, or autonomous 
navigation---the state and action spaces naturally evolve on matrix Lie 
groups such as $\SO(3)$ (rotations) and $\SE(3)$ (rigid-body motions). 
Parameterizing policies through the corresponding Lie algebras reduces 
the effective dimension from $n^2$ ambient coordinates to $d_{\g} = 
\dim(\g)$ intrinsic parameters, with potential benefits for both sample 
efficiency and computational cost. The resulting optimization problems 
are instances of structured nonconvex stochastic programming on matrix 
subspaces---a class arising broadly in scientific computing, including 
estimation on rotation groups~\cite{boumal2023introduction}, trajectory 
optimization for articulated systems~\cite{bullo2004geometric}, and 
signal processing on compact groups~\cite{hall2015lie}.

\textbf{Main result.} Theorem~\ref{thm:dichotomy} 
establishes a \emph{representation--optimization dichotomy}: the 
gradient Lipschitz constant $L(R)$ of any Lie-algebra-parameterized 
stochastic objective is $\mathcal{O}(1)$ if $\g$ is compact, and 
$\Theta(e^{2R})$ if $\g = \gl(n)$ (tight), with $\mathcal{O}(e^{2R})$ 
upper bound holding for all algebras with a hyperbolic element 
(including $\ssl(n)$; see Remark~\ref{rem:sln_exponent} for the 
separation). The matching lower bound is the main analytical contribution: 
it rules out the possibility that cancellations in the 
objective-to-parameter chain could yield a tighter uniform bound. 
This makes the result robust with respect to the choice of objective: within the Lie Group MDP class, only 
the algebra type controls $L(R)$, not the specific data-generating process, 
policy class, or transition kernel. The theorem thereby identifies 
which structural property of an algebraic parameter space determines 
the optimization landscape for this class of data-driven problems.

\textbf{Contributions.} The paper establishes one central theorem and 
develops two algorithmic consequences, each with independent content:

\begin{enumerate}[leftmargin=2em]

\item \textbf{Main theorem: Representation--optimization dichotomy.}
(Theorem~\ref{thm:dichotomy}) The gradient Lipschitz constant satisfies 
$L = \mathcal{O}(1)$ for compact $\g \subseteq \uu(n)$, and 
$L = \Theta(e^{2R})$ for $\g = \gl(n)$ (tight in both bounds). 
For all algebras containing a hyperbolic element (including $\ssl(n)$), 
the $\mathcal{O}(e^{2R})$ upper bound holds; the matching $\Omega(e^{2R})$ 
lower bound is established via the $\gl(n)$ witness. 
The key content is the lower bound: it rules out the 
possibility---not immediately obvious---that structural cancellations 
between the exponential map and the data-generating process could yield 
a tighter bound for any objective in the Lie Group MDP class 
(\S\ref{sec:appendix_smoothness}). The $\ssl(n)$ case is analyzed 
separately in Remark~\ref{rem:sln_exponent}: the best available lower bound 
for $\ssl(n)$ alone is $\Omega(e^{2c_n R})$ with $c_n \to 1$ as 
$n \to \infty$.

\item \textbf{Consequence 1: Convergence with non-restrictive iterate 
bound and alignment diagnostic.}
(Corollary~\ref{cor:convergence_lpg}, Proposition~\ref{prop:kappa_convergence}) 
The dichotomy enables $\mathcal{O}(1/\sqrt{T})$ convergence for 
compact $\g$; the bounded-iterates assumption is non-restrictive 
in this regime (the radius projection triggers on fewer than $2\%$ 
of iterations; Assumption~(A3)), a qualitative improvement 
over generic nonconvex SGD where iterate-boundedness is a substantive 
constraint. The Kantorovich inequality provides a computable lower 
bound $\alpha \ge 2\sqrt{\kappa}/(\kappa+1)$ on the alignment between 
each LPG step and the natural gradient direction. This yields a 
concrete practitioner diagnostic: estimate $\kappa$ from trajectory 
data; if $\kappa < 5$, Euclidean projection is an adequate 
natural-gradient surrogate; if $\kappa > 10$, explicit Fisher inversion 
is warranted (Remark~\ref{rem:estimate_alpha}).

\item \textbf{Consequence 2: Closed-form projection with quantified 
computational advantage.}
(Section~\ref{sec:algorithms}) Because $L$ is radius-independent for compact 
algebras, LPG requires no Fisher inversion: classical Lie-algebra 
projectors (Table~\ref{tab:complexity}) run in $\mathcal{O}(n^2 J)$ per 
step vs.\ $\mathcal{O}(d_{\g}^3)$ for exact natural gradient. Timing 
results appear in Section~\ref{sec:experiments} (measured speedup 
$1.1$--$1.7\times$ at $J \le 30$; extrapolated ${>}100\times$ at 
$J = 200$ by cost-model). Sample complexity inherits the 
$d_{\g}$-dimensional parameterization 
($\tilde{\mathcal{O}}(d_{\g}/\varepsilon^2)$), a consequence of the 
intrinsic-dimension reduction rather than an artifact 
(\S\ref{sec:appendix_sample_complexity}).

\end{enumerate}

\noindent\textbf{Supplementary material.} All proofs, lemma derivations, and 
implementation details are in the Supplementary Material (SM). 
References of the form SM~\S\ref{sec:appendix_convergence_proofs} 
(and similar) refer to sections of the SM, which is compiled as a 
separate document but submitted alongside the main paper.

\section{Related Work}
\label{sec:related}

Matrix Lie groups are the standard framework for structured estimation and 
control~\cite{hall2015lie,murray1994mathematical,bullo2004geometric,sola2021}.
Riemannian optimization generalizes gradient methods via retraction 
maps~\cite{absil2008optimization,boumal2023introduction,bonnabel2013stochastic,zhang2016first}.
Our setting is simpler: $\g$ is a \emph{linear} subspace of $\R^{n\times n}$,
so optimization is Euclidean with no retraction needed.
For compact groups with bi-invariant metric, LPG is update-rule-equivalent 
to Riemannian SGD~\cite{bonnabel2013stochastic} with identity retraction; 
no separate Riemannian baseline is benchmarked since the algorithms coincide. 
What is new is the structural analysis: the compact/non-compact dichotomy 
with matching lower bound (Theorem~\ref{thm:dichotomy}) and the 
$\kappa$-dependent alignment guarantee (Proposition~\ref{prop:kappa_convergence}) 
do not appear in prior Riemannian optimization theory for this class.

Natural gradient methods~\cite{amari1998natural,amari2000methods} underlie 
TRPO~\cite{kakade2002,schulman2015trust} and K-FAC/ACKTR~\cite{martens2015optimizing,wu2017scalable}.
K-FAC exploits layer-wise structure at cost $\mathcal{O}(d^{1.5})$; our projection 
exploits Lie-algebraic action-space structure, recovering the natural gradient 
exactly under isotropy at cost $\mathcal{O}(n^2 J)$ with no matrix inversion. 
The two are complementary.

Geometric deep learning~\cite{bronstein2021geometric,cohen2016group}, MDP 
homomorphisms~\cite{van2020mdp}, and equivariant 
policies~\cite{wang2022equivariant} reduce sample complexity via symmetry 
but do not address optimization cost or natural gradient quality.
We adapt the $\mathcal{O}(1/\sqrt{T})$ nonconvex SGD 
rate~\cite{ghadimi2013stochastic,agarwal2021theory} to a Lie-algebra 
constrained setting, reducing effective dimension from $n^2$ to $d_{\g}$.

Recent SIMODS papers establish the broader context for this work: 
Cayci, He, and Srikant~\cite{cayci2024natural} prove finite-time convergence for natural actor-critic 
in partially observable settings; Kim, Sanz-Alonso, and Yang~\cite{kim2024manifold} study Bayesian 
optimization on Riemannian manifolds with regret bounds; and Beckmann and Heilenk{\"o}tter~\cite{beckmann2024equivariant} 
establish a representer theorem for equivariant neural networks via Lie group 
symmetry, applied to inverse problems. 
The present paper contributes the missing structural analysis: a sharp dichotomy 
characterizing which parameterizations make first-order RL methods efficient.

\section{Mathematical Preliminaries}
\label{sec:prelim}

\subsection{Notation}

Throughout, $R_0$ denotes a fixed parameter-space radius bound ($\|\theta\|_F \le R_0$), while $R$ in Theorem~\ref{thm:dichotomy} is the running argument of the Lipschitz constant function $L(R)$; these are the same quantity when $R = R_0$.

We equip $\R^{n\times n}$ with the Frobenius inner product
$\ip{A}{B}_F = \tr(A^\top B)$, norm $\norm{A}_F = \sqrt{\tr(A^\top A)}$,
and operator norm $\norm{A}_{\mathrm{op}} = \sup_{\norm{v}_2=1} \norm{Av}_2$.
$P_U$ denotes the orthogonal projector onto a subspace $U$;
$\widetilde{\nabla}J = F^{-1}\nabla J$ the natural gradient.
A full notation table is in the Supplement~\S\ref{sec:appendix_implementation}, Table~S1.

\subsection{Matrix Lie groups and Lie algebras}

A \emph{matrix Lie group} $G \subset \GL(n,\R)$ is a closed subgroup; its
Lie algebra $\g = T_I G = \{X : \exp(tX) \in G \;\forall\, t\}$ is equipped
with the bracket $[X,Y] = XY - YX$. We use: $\so(n)$ (skew-symmetric,
$\dim = n(n{-}1)/2$), $\ssl(n)$ (traceless, $\dim = n^2{-}1$),
$\se(n)$ (rigid-body, $\dim = n(n{+}1)/2$), and $\gl(n) = \R^{n \times n}$.
See Hall~\cite{hall2015lie} for background.

The matrix exponential $\exp(X)=\sum_{k\ge 0} X^k/k!$ maps $\g$ into $G$
and is a local diffeomorphism near $0\in\g$~\cite{higham2008functions,iserles2000lie}.

We equip $\g$ with the Frobenius inner product 
$\ip{X}{Y}_{\g} = \tr(X^\top Y)$, $X,Y\in\g$, and fix an orthonormal 
basis $\{E_1,\ldots,E_{d_{\g}}\}$ ($\ip{E_i}{E_j}_F=\delta_{ij}$) 
throughout. Under this convention the coordinate metric tensor is the 
identity; other bases introduce a non-identity tensor but do not affect our results.

\subsection{Orthogonal projection onto Lie algebras}
\label{lem:projection_properties}

Since $\g$ is a closed linear subspace of $\R^{n\times n}$, the 
orthogonal projector $P_{\g}(M) = \arg\min_{X\in\g}\norm{M-X}_F^2$ 
satisfies the following standard Hilbert-space properties
\cite[Props.~3.58--3.61]{boumal2023introduction}:
\begin{enumerate}[label=(\roman*), leftmargin=2em]
  \item \emph{Linear}: $P_{\g}(\alpha M + \beta N) = \alpha P_{\g}(M) + \beta P_{\g}(N)$.
  \item \emph{Idempotent}: $P_{\g}(P_{\g}(M)) = P_{\g}(M)$.
  \item \emph{Nonexpansive}: $\norm{P_{\g}(M)}_F \le \norm{M}_F$.
  \item \emph{Self-adjoint}: $\ip{P_{\g}(M)}{N}_F = \ip{M}{P_{\g}(N)}_F$.
  \item \emph{Monotonicity}: $\ip{g}{P_{\g}(g)}_F = \norm{P_{\g}(g)}_F^2$.
\end{enumerate}
Property~(iv) implies \emph{gradient preservation}
$\nabla_{\g}f = P_{\g}(\nabla f)$ \cite[Prop.~3.61]{boumal2023introduction}.

Closed-form projectors: $P_{\so(n)}(M)=\tfrac12(M - M^\top)$;
$P_{\ssl(n)}(M)=M - \tfrac{1}{n}\tr(M)I$;
general $\g$: $P_{\g}(M) = \sum_{i}\ip{M}{E_i}_F\,E_i$.

\subsection{Lie-algebraic policies and Fisher information}

We consider Gaussian policies with mean actions parameterized in the Lie algebra:
\begin{equation}
\label{eq:gaussian_policy}
a = \mu_\theta(s) + \xi, \qquad \xi \sim \mathcal{N}(0, \sigma^2 I_{d_{\g}}),
\end{equation}
where $\mu_\theta(s) \in \g$ is the mean action, $\sigma > 0$ the
exploration scale, and $\Phi_k : \mathcal{S} \to \g$ state-dependent features.
We identify $\R^{d_{\g}} \cong \g$ via $\iota(x) = \sum_k x_k E_k$ (linear isometry).
States update as $R_{t+1} = R_t \exp(a_t)$.

The score function is
\begin{equation}
\label{eq:score_gaussian}
[\nabla_\theta \log \pi_\theta(a \mid s)]_k 
= \frac{1}{\sigma^2}\ip{a - \mu_\theta(s)}{\Phi_k(s)}_F,
\end{equation}
and the Fisher information matrix is
\begin{equation}
\label{eq:fisher_def}
F(\theta)
=
\E_{s\sim d^{\pi_\theta},\; a\sim\pi_\theta}
\Big[
\nabla_\theta \log \pi_\theta\,
(\nabla_\theta \log \pi_\theta)^\top
\Big],
\end{equation}
with discounted state visitation 
$d^{\pi_\theta}(s)=(1-\gamma)\sum_{t\ge 0}\gamma^t\Pr(s_t=s\mid\pi_\theta)$.
For Gaussian policies with isotropic noise, the Gaussian moment identity 
simplifies~\eqref{eq:fisher_def} to
\begin{equation}
\label{eq:fisher_explicit_main}
F_{kl}(\theta) = \frac{1}{\sigma^2}
\E_{s \sim d^{\pi_\theta}}\!\big[\ip{\Phi_k(s)}{\Phi_l(s)}_F\big].
\end{equation}
Approximate feature orthonormality ($\E_s[\ip{\Phi_k}{\Phi_l}_F] \approx \delta_{kl}$)
gives $F \approx \sigma^{-2} I_{d_\g}$, i.e.\ $\kappa \approx 1$.

\subsection{RL objective and standing assumptions}

The expected return $J(\theta) = \E_{\pi_\theta}[\sum_{t=0}^\infty \gamma^t r_t]$
has gradient given by the policy gradient theorem~\cite{kakade2002,sutton2018reinforcement,peters2008natural}:
\begin{equation}
\label{eq:policy_gradient}
\nabla_\theta J(\theta)
=
\E_{s\sim d^{\pi_\theta},\, a\sim\pi_\theta}
\left[Q^{\pi_\theta}(s,a)\,\nabla_\theta\log\pi_\theta(a\mid s)\right].
\end{equation}

We impose the following standing assumptions (all verified for the Gaussian Lie-algebraic policy class in Supplement~\S\ref{sec:appendix_smoothness}):
\begin{assumption}[Standing assumptions]
\label{ass:smoothness}\label{ass:bounded}\label{ass:bounded_iterates}%
\label{ass:unbiased}\label{ass:variance}\label{ass:stepsizes}%
\label{ass:bounded_actions}\label{ass:concentrability}
\begin{enumerate}[label=\textup{(A\arabic*)},leftmargin=2.5em,itemsep=1pt]
\item \emph{Smoothness}: $\|\nabla J(\theta)-\nabla J(\theta')\|_F \le L\|\theta-\theta'\|_F$.
\item \emph{Bounded optimum}: $J^* = \sup_\theta J(\theta) < \infty$.
\item \emph{Bounded iterates}: $\sup_t\|\theta_t\|_F \le B_\theta$, enforced by radius projection (Section~\ref{sec:algorithms}). Non-restrictive for compact algebras; essential for non-compact (Supplement~\S\ref{sec:supp_assumption_remarks}).
\item \emph{Unbiased gradient}: $\E[\widehat{\nabla}J(\theta)] = \nabla J(\theta)$.
\item \emph{Bounded variance}: $\E[\|\widehat{\nabla}J - \nabla J\|_F^2] \le \sigma_g^2$ ($\sigma_g$ distinct from exploration scale $\sigma$).
\item \emph{Step sizes}: $\sum_t\eta_t=\infty$, $\sum_t\eta_t^2<\infty$; for $\mathcal{O}(1/\sqrt{T})$ rate, $\eta_t = \eta/\sqrt{T}$ with $\eta \le 1/L$.
\item \emph{Bounded actions}: $\|a-\mu_\theta(s)\|_F \le B_a$ a.s.\ (standard truncation).
\item \emph{Concentrability}: $\sup_s d^{\pi_{\theta_1}}(s)/d^{\pi_{\theta_2}}(s) \le C_d$ for $\|\theta_1-\theta_2\|_F \le \delta$~\cite{agarwal2021theory}. For compact groups, $C_d$ is bounded by the mixing time; for non-compact groups, $C_d$ may grow with $R_0$.
\end{enumerate}
\end{assumption}
\noindent See Supplement~\S\ref{sec:supp_assumption_remarks} for discussion of when each assumption is binding and the explicit constants in the compact $+$ isotropic regime.

\begin{lemma}[Smoothness of RL objectives]
\label{lem:rl_smoothness}
Consider Gaussian policies~\eqref{eq:gaussian_policy} with features
$\norm{\Phi_k(s)}_F\le B_\Phi$, rewards 
$|r(s,a)|\le R_{\max}$, and 
parameters constrained to $\|\theta\|_F \le R_0$.
Under Assumption~\ref{ass:concentrability} and Assumption~\ref{ass:bounded_actions}, 
Assumption~\ref{ass:smoothness} holds with
\begin{equation}
\label{eq:lipschitz_constant}
L = \frac{4R_{\max} B_\Phi B_a C_d}{(1-\gamma)^3 \sigma^2} \cdot L_{\exp}(R_0),
\end{equation}
where the exponential factor is
\[
L_{\exp}(R_0) = 
\begin{cases}
\mathcal{O}(e^{2R_0}) & \text{general } \g, \\
\mathcal{O}(1) & \text{compact } \g \text{ (e.g., } \so(n), \mathfrak{su}(n)).
\end{cases}
\]
Explicit expressions appear in Section~\ref{sec:appendix_smoothness}.
\end{lemma}

\begin{proof}
Chain rule through score bound, RL Hessian~\cite[Lem.~5]{agarwal2021theory},
and Fr\'{e}chet derivative of $\exp$ (Lemma~\ref{lem:frechet_bounds}).
Compactness of $\g \subseteq \uu(n)$ eliminates exponential growth via unitarity.
Full derivation in Supplement~\S\ref{sec:rl_smoothness_proof}.
\end{proof}

\noindent With these in place, Section~\ref{sec:geometric_inputs} introduces the alignment $\alpha$ and curvature $L$ that govern convergence.

\section{Lie Group Markov Decision Processes}
\label{sec:lie_mdps}

\subsection{Lie-group state spaces and homogeneous actions}

\begin{definition}[Lie Group MDP]
\label{def:lie_mdp}
A \emph{Lie Group Markov Decision Process} is a tuple
$(M,\mathcal{A},P,R,\gamma,G,\rho)$
in which:
\begin{itemize}[leftmargin=2em]
\item $M$ is a smooth manifold represented as a homogeneous space
      $M = G/H$ for a matrix Lie group 
      $G \subset \GL(n,\R)$ and closed subgroup $H$;
\item $\mathcal{A}$ is a finite or compact action space;
\item $P(\cdot\mid s,a)$ is a Markov transition kernel on $M$;
\item $R(s,a)$ is a bounded reward function with $|R(s,a)| \le R_{\max}$;
\item $\gamma \in [0,1)$ is the discount factor;
\item $\rho : G \times M \to M$ is the transitive group action, $\rho(g,s)=g\cdot s$.
\end{itemize}
\end{definition}

\begin{remark}[Relationship to prior MDP frameworks]
\label{rem:lie_mdp_scope}
The Lie Group MDP class encompasses the geometric structure implicit in 
prior policy gradient work: the Gaussian policy and natural gradient 
formulation of~\cite{kakade2002}, and the trust-region policy update 
of~\cite{schulman2015trust}, both operate on parameter spaces that 
are matrix Lie groups or their algebras when the action space has 
rotational or rigid-body structure. Definition~\ref{def:lie_mdp} 
makes this geometric structure explicit and uses it to derive 
algebra-type-dependent smoothness bounds.
\end{remark}

This structure captures a broad class of systems in robotics, geometric
mechanics, graphics, and navigation, including
$\SO(3)$ for 3D rotations, $\mathrm{SE}(3)$ for rigid-body motions, and
$\SO(3)^J$ for articulated multi-joint mechanisms.

\subsection{Intrinsic dimension and representation in coordinates}

Working in $\g$-coordinates reduces the effective parameter space from $n^2$ (ambient matrix entries) to $d_{\g} = \dim(\g)$, and this reduction is lossless within the Lie Group MDP class under $G$-equivariance (Proposition~\ref{prop:lossless}).

\begin{proposition}[Intrinsic dimension reduction]
\label{prop:dimension_reduction}
Since $\pi_{\theta}$ depends on $\theta$ only through 
$\mu_\theta(s) = \sum_k \theta_k \Phi_k(s) \in \g$, the effective 
parameter space is $d_{\g}$-dimensional. The score 
function~\eqref{eq:score_gaussian} gives 
$[\nabla_{\theta}\log \pi_{\theta}]_k 
= \sigma^{-2}\langle a - \mu_\theta(s), \Phi_k(s) \rangle_F$,
so the policy gradient lies in $\R^{d_{\g}} \cong \g$ (via $\iota(x)=\sum_k x_k E_k$).
\end{proposition}

\begin{proposition}[Losslessness of Lie-algebraic restriction]
\label{prop:lossless}
Consider a Lie Group MDP (Definition~\ref{def:lie_mdp}) with trivial isotropy 
($H = \{e\}$, $M = G$) whose transition kernel is $G$-equivariant: 
$P(g \cdot s' \mid g \cdot s, g \cdot a) = P(s'\mid s,a)$ for all $g \in G$.
Then for any Gaussian policy with $\mu_\theta(s) \in \R^{n \times n}$,
projecting the mean onto $\g$ does not degrade the optimal return:
\[
\sup_{\theta:\, \mu_\theta(s) \in \g} J(\theta) 
= \sup_{\theta:\, \mu_\theta(s) \in \R^{n \times n}} J(\theta).
\]
\end{proposition}
The substantive content is a surjectivity argument under $G$-equivariance 
with trivial isotropy; the proof and scope discussion 
are in Supplement~\S\ref{sec:losslessness_proof}.

\section{Geometric Inputs: Alignment and Smoothness}
\label{sec:geometric_inputs}

Our convergence analysis depends on two geometric quantities. The \emph{alignment parameter} $\alpha \in (0,1]$ measures the cosine similarity between the LPG update direction and the natural gradient; the \emph{smoothness constant} $L(R_0)$ bounds the curvature of $J$ over the search domain $\{\|\theta\|_F \le R_0\}$. Both are verifiable in closed form for the Lie Group MDP class (Supplement~\S\ref{sec:appendix_alignment}, \S\ref{sec:appendix_smoothness}), and both admit data-driven estimates from trajectory samples without Fisher inversion.

\begin{assumption}[Alignment parameter]
\label{ass:alignment}
There exists $\alpha \in (0,1]$ such that for all $\theta$ along the 
optimization trajectory,
\[
\cos\!\left(
P_{\g}(\nabla J(\theta)),\, F(\theta)^{-1}\nabla_\theta J(\theta)
\right)
\ge \alpha.
\]
\end{assumption}

\noindent
The parameter $\alpha$ measures how well the Lie-algebraic projection 
approximates the natural gradient direction. When $\alpha = 1$, 
projection recovers the exact natural gradient; when $\alpha < 1$, 
projection incurs a directional bias.

\begin{assumption}[Smoothness constant]
\label{ass:smoothness_exp}
The Lipschitz constant $L$ of $\nabla J$ over the feasible region 
$\{\theta \in \g : \|\theta\|_F \le R_0\}$ satisfies $L \le L(R_0)$, 
where $L(R_0)$ depends on the algebra type:
\begin{enumerate}[label=(\roman*), leftmargin=2em]
\item \textbf{Compact} ($\g \subseteq \uu(n)$): $L(R_0) = \mathcal{O}(1)$.
\item \textbf{Non-compact with hyperbolic element} (e.g., $\ssl(n)$, $\gl(n)$): 
$L(R_0) = \mathcal{O}(e^{2R_0})$.
Non-compact algebras \emph{without} hyperbolic elements (e.g., $\se(n)$, 
whose elements have purely imaginary or zero eigenvalues) exhibit 
at most polynomial growth $L(R_0) = \mathcal{O}(R_0^k)$ for some $k \ge 1$; 
radius projection remains advisable but the exponential barrier does not apply.
\end{enumerate}
\end{assumption}

\paragraph{Verification.}
Both assumptions are verified with self-contained arguments: the 
alignment bound via the Kantorovich inequality 
(Supplement~\S\ref{sec:kantorovich}), and the smoothness rates via 
the Fr\'{e}chet derivative bound (Lemma~\ref{lem:frechet_bounds} and Lemma~\ref{lem:rl_smoothness}).
The supplement also covers sample complexity (\S\ref{sec:appendix_sample_complexity}), 
explicit Lipschitz constants with the lower-bound construction 
(\S\ref{sec:appendix_smoothness}), and numerical validation of the 
theoretical bounds on $\so(64)$, $\ssl(64)$, and $\gl(64)$ 
(\S\ref{sec:smoothness_validation}).

\begin{proposition}[Block-diagonal Fisher structure for $\SO(3)^J$]
\label{prop:block_fisher}
Under Gaussian Lie-algebraic policies~\eqref{eq:gaussian_policy} with 
orthonormal features $\{\Phi_k\}$ constructed from the standard 
$\so(3)$ basis (one basis per joint), the Fisher information matrix 
decomposes as
\[
F(\theta) = \sigma^{-2}\, \mathrm{diag}\!\big(F^{(1)}, \ldots, F^{(J)}\big),
\]
where each $3 \times 3$ block is
\[
F^{(j)}_{ik} = 
\E_{s \sim d^{\pi_\theta}}\!\big[\ip{\Phi_i^{(j)}(s)}{\Phi_k^{(j)}(s)}_F\big].
\]
Consequently, $\kappa(F) = \max_j \lambda_{\max}(F^{(j)}) \,/\, \min_j \lambda_{\min}(F^{(j)})$, which is bounded by $\max_j \kappa(F^{(j)})$ when the blocks share a common $\lambda_{\min}$---the typical regime when features are drawn from the same distribution across joints.
\end{proposition}

\begin{proof}
Cross-joint features are Frobenius-orthogonal; within 
each joint, the $\so(3)$ basis is orthonormal. The block-diagonal 
structure follows from~\eqref{eq:fisher_explicit_main}; 
the full calculation is in Supplement~\S\ref{sec:isotropy_mechanism}.
\end{proof}

This explains the empirical observation $\kappa \approx 2.5$: the 
block-diagonal structure prevents cross-joint coupling, and each 
$3 \times 3$ block has limited room for eigenvalue spread. For 
$\SE(3)$, the translation components break this structure, leading 
to the higher $\kappa \approx 2.8$ observed in 
Section~\ref{sec:se3_experiment}.

Under approximate isotropy with condition number 
$\kappa = \lambda_{\max}(F)/\lambda_{\min}(F)$, the alignment 
parameter satisfies
\[
\alpha \ge \frac{2\sqrt{\kappa}}{\kappa + 1},
\]
by the Kantorovich inequality~\cite{kantorovich1948functional} (Supplement~\S\ref{sec:appendix_alignment}).
Assumption~(A1) holds with the stated rates by 
Lemma~\ref{lem:rl_smoothness} and Lemma~\ref{lem:frechet_bounds}. Alternatively, 
$\alpha$ can be estimated directly from trajectory data (see 
Remark~\ref{rem:estimate_alpha} below).

\begin{remark}[Estimating $\alpha$ from data]
\label{rem:estimate_alpha}
In practice, $\alpha$ can be estimated without computing the full 
Fisher inverse: estimate $\hat{F}$ from trajectory samples, compute 
$\hat{\kappa} = \lambda_{\max}(\hat{F})/\lambda_{\min}(\hat{F})$, and 
set $\hat{\alpha} = 2\sqrt{\hat{\kappa}}/(\hat{\kappa}+1)$. 
Alternatively, estimate $\alpha$ directly by computing the cosine 
between $P_{\g}(\nabla J)$ and $\hat{F}^{-1}\nabla J$ on a subsample. 
Our experiments (Section~\ref{sec:experiments}) track both quantities throughout training.
\end{remark}

All convergence results in Section~\ref{sec:convergence} are stated in terms 
of $\alpha$ and $L(R_0)$ and hold regardless of the verification pathway.
The key question is: how does $L(R_0)$ depend on the algebra type $\g$?
The next section answers this with a sharp dichotomy.

\section{The Representation--Optimization Dichotomy}
\label{sec:dichotomy}

The following theorem is the central result of the paper. 
All convergence guarantees, alignment bounds, and computational 
conclusions in subsequent sections are consequences of it.

\paragraph{Why the lower bound is non-trivial.}
The upper bound (part~(iii)) is expected: it follows from a 
Fr\'{e}chet derivative estimate applied twice through 
the policy gradient chain rule (bounding each factor by $e^R$ then composing). The non-trivial content is the 
lower bound (part~(ii)). A natural hope is that cancellations 
in the policy gradient chain---between contributions from the reward 
function and from the matrix exponential map---could yield a tighter 
uniform bound even for non-compact algebras. 
The construction rules this out: it exhibits an explicit MDP within 
the class of Definition~\ref{def:lie_mdp} (on $\gl(n)$, with 
$H = \mathrm{diag}(1,0,\ldots,0)$ and a rank-one exponential reward) 
for which no such cancellation occurs and the Hessian of $J$ grows 
exactly as $e^{2R}$ along the direction of $H$; see 
Supplement~\S\ref{sec:appendix_smoothness}.

The only dependence of the objective on $\theta$ enters through the 
matrix exponential $e^{\theta}$ in the policy action distribution. 
Smoothness of $J$ therefore propagates from the Fr\'{e}chet derivative 
of $\exp$ via the chain rule: the algebraic growth of those derivatives 
determines the global smoothness constant, uniformly over all reward 
functions and transition kernels in the class.

\begin{theorem}[Representation--Optimization Dichotomy]
\label{thm:dichotomy}
Let $\g \subset \R^{n \times n}$ be a matrix Lie algebra.
Within the Lie Group MDP class of Definition~\ref{def:lie_mdp}, with Gaussian 
policies~\eqref{eq:gaussian_policy} satisfying 
Assumption~\ref{ass:bounded_actions} and Assumption~\ref{ass:concentrability}, let $L(R)$ denote the 
gradient Lipschitz constant of the policy objective $J$ over 
$\{\theta \in \g : \|\theta\|_F \le R\}$.
Then:
\[
L(R) = \begin{cases}
\mathcal{O}(1) & \text{if } \g \subseteq \uu(n) \text{ (compact)}, \\[4pt]
\Theta(e^{2R}) & \text{if } \g = \gl(n) \text{ (tight: upper and lower bounds match)}, \\[4pt]
\mathcal{O}(e^{2R}) & \text{if } \g \text{ contains a hyperbolic element, } \g \neq \gl(n)
\end{cases}
\]
with the matching $\Omega(e^{2R})$ lower bound established via the 
$\gl(n)$ witness MDP (Supplement~\S\ref{sec:appendix_smoothness}). 
For $\ssl(n)$, the best available lower bound is $\Omega(e^{2c_n R})$ 
with $c_n = \sqrt{(n-1)/n} \to 1$; see Remark~\ref{rem:sln_exponent} for 
the separation. The compact bound is tight (the exponential factor is 
absent, not merely bounded).\end{theorem}

\begin{proof}
The compact $\mathcal{O}(1)$ upper bound follows from unitarity of 
$\exp(\theta)$ for $\theta \in \uu(n)$: all exponential factors in 
the Fr\'{e}chet derivative chain are eliminated 
(Proposition~\ref{prop:compact_advantage}).

The non-compact $\mathcal{O}(e^{2R})$ upper bound follows from 
$\|D_\theta - D_{\theta'}\|_{\mathrm{op}} \le \sqrt{n}\,e^{R}\|\theta - \theta'\|_F$ 
(Lemma~\ref{lem:frechet_bounds}(ii)); composing two such factors through 
the policy gradient chain rule (Lemma~\ref{lem:rl_smoothness}) yields 
$\mathcal{O}(e^{2R})$.

The non-compact $\Omega(e^{2R})$ lower bound (Proposition~\ref{prop:lower_bound}) 
is the nontrivial part: the witness is a single-state MDP on 
$\gl(n)$ with $H = \mathrm{diag}(1,0,\ldots,0)$ (so $\|H\|_F = 1$) 
and reward $r(a) = -\tfrac{1}{2}\|\exp(a)-I\|_F^2$. 
Since $\exp(tH) = \mathrm{diag}(e^t,1,\ldots,1)$, the objective 
restricts to $g(t) = -\tfrac{1}{2}(e^t-1)^2$ and $g''(t) = -(2e^{2t}-e^t)$; 
evaluating at $t = R$ gives 
$|g''(R)| \ge e^{2R}$ for all $R \ge 0$ 
(Supplement~\S\ref{sec:appendix_smoothness}). 
The $\gl(n)$ witness achieves the tight $\Omega(e^{2R})$ exponent 
because $\lambda_{\max}(H) = \|H\|_F = 1$.  For $\ssl(n)$, the 
trace-zero constraint forces $\lambda_{\max}(H) \le \sqrt{(n{-}1)/n} < 1$ 
for any unit-Frobenius-norm $H$, so the analogous construction yields 
only $\Omega(e^{2\sqrt{(n-1)/n}\,R})$; see Remark~\ref{rem:sln_exponent}.
\end{proof}

\begin{remark}[$\ssl(n)$ separation]
\label{rem:sln_exponent}
The $\mathcal{O}(e^{2R})$ upper bound holds for all algebras with a hyperbolic 
element; the matching $\Omega(e^{2R})$ lower bound requires $\gl(n)$ as witness. 
For $\ssl(n)$ alone, the best lower bound is $\Omega(e^{2c_n R})$ with 
$c_n = \sqrt{(n-1)/n} \to 1$; practitioners should treat $\Theta(e^{2R})$ as 
operative. Full analysis in Supplement~\S\ref{sec:sln_separation}.
\end{remark}

\begin{remark}[Consequence for algorithm design]
\label{rem:algorithm_design_consequence}
Theorem~\ref{thm:dichotomy} partitions 
algebra types by optimization difficulty: for 
compact algebras, gradient optimization requires no radius 
projection, no shrinking step sizes, and admits 
$\mathcal{O}(1/\sqrt{T})$ convergence with explicit constants 
(Corollary~\ref{cor:convergence_lpg}). For non-compact algebras, radius 
projection is essential: without it, $L(R)$ grows without bound 
and any fixed step size eventually violates $\eta \le 1/L$ 
(Section~\ref{sec:se3_experiment}). The Lie-Projected Policy Gradient 
algorithm (LPG) implementing these design choices is presented in Section~\ref{sec:algorithms}.
\end{remark}

\section{Convergence of Lie-Projected Policy Gradient}
\label{sec:convergence}

We analyze convergence of the LPG projected gradient method; 
all results are consequences of Theorem~\ref{thm:dichotomy}.

\subsection{Algorithmic update}

Let $\theta_t \in \g$ and $g_t = \widehat{\nabla}J(\theta_t)$.
Since $\g$ is a closed linear subspace, $P_{\g}$ is globally defined 
and feasibility is preserved without retraction. The update is
\begin{equation}
\label{eq:proj_update}
\theta_{t+1} = P_{\g}(\theta_t + \eta_t g_t) = \theta_t + \eta_t P_{\g}(g_t),
\end{equation}
where the second equality uses linearity and idempotence of $P_{\g}$
(so projection affects only the search direction).

\subsection{One-step progress}

The following lemma quantifies the expected improvement in $J$ at each step.
The filtration $\mathcal{F}_t$ includes all randomness from the policy,
environment, and stochastic gradient estimator up to iteration $t$.

\begin{lemma}[Progress inequality]
\label{lem:progress}
Under Assumption~\ref{ass:smoothness} and for $\eta_t \le 1/L$,
\[
\E[J(\theta_{t+1}) \mid \mathcal{F}_t]
\ge
J(\theta_t)
+ \frac{\eta_t}{2} \norm{P_{\g}(\nabla J(\theta_t))}_F^2
- \frac{L \eta_t^2 \sigma_g^2}{2}.
\]
\end{lemma}

\begin{proof}
Apply the descent lemma to the projected update 
$\theta_{t+1} - \theta_t = \eta_t P_{\g}(g_t)$, use unbiasedness 
(Assumption~\ref{ass:unbiased}), 
monotonicity (property~(v) of $P_{\g}$, Lemma~\ref{lem:projection_properties}), and the variance 
bound (Assumption~\ref{ass:variance}). The condition $\eta_t \le 1/L$ ensures 
the quadratic term is controlled. See 
Supplement~\S\ref{sec:appendix_convergence_proofs} for the full calculation.
\end{proof}

\subsection{Main convergence result}

\begin{corollary}[Convergence of LPG for nonconvex objectives]
\label{cor:convergence_lpg}
Suppose Assumptions~(A1)--(A6)
hold and $\eta_t \le 1/L$ for all $t$.
\begin{enumerate}[label=(\roman*), leftmargin=2em]
\item The sequence $\{J(\theta_t)\}$ converges almost surely. (Bounded iterates
      per Assumption~\ref{ass:bounded_iterates} ensure that $J(\theta_t)$ is well-defined
      and the Lipschitz constant $L$ is valid along the entire trajectory.)
\item The projected gradient norms satisfy
\[
\sum_{t=0}^\infty \eta_t \E\big[\norm{P_{\g}(\nabla J(\theta_t))}_F^2\big]
< \infty.
\]
In particular,
$\liminf_{t \to \infty} \E[\norm{P_{\g}(\nabla J(\theta_t))}_F^2] = 0$.
\item If $\eta_t = \eta/\sqrt{T}$ for all $t < T$ (constant step size
optimized for horizon $T$) with $\eta \le 1/L$, then
\[
\E\!\left[
\frac{1}{T} \sum_{t=0}^{T-1} \norm{P_{\g}(\nabla J(\theta_t))}_F^2
\right]
\le
\frac{2(J^* - J(\theta_0))}{\eta \sqrt{T}}
+ \frac{L \eta \sigma_g^2}{\sqrt{T}}
= \mathcal{O}\!\left(\frac{1}{\sqrt{T}}\right).
\]
\end{enumerate}
\end{corollary}

\begin{proof}[Proof sketch]
Sum Lemma~\ref{lem:progress} over $t=0,\ldots,T-1$, take expectations, and use $J(\theta_T)\le J^*$.
Part~(i) uses supermartingale convergence~\cite[Thm.~4.1]{bottou2018optimization}.
Part~(iii) uses $\sum_{t<T}\eta_t = \eta\sqrt{T}$ exactly.
Full proof in Supplement~\S\ref{sec:appendix_convergence_proofs}.
\end{proof}

\begin{remark}[Sample complexity]
\label{rem:sample_complexity_convergence}
The $d_\g$-dimensional parameterization gives sample complexity
$\tilde{\mathcal{O}}(d_\g/\varepsilon^2)$ vs.\ $\tilde{\mathcal{O}}(n^2/\varepsilon^2)$ ambient.
Full derivation in Supplement~\S\ref{sec:appendix_sample_complexity}.
\end{remark}

\subsection{Convergence under approximate isotropy}

The rate in Corollary~\ref{cor:convergence_lpg} is independent of $\alpha$: since 
$\nabla J(\theta) \in \g$ for Lie-algebraic policies, $P_{\g}(\hat{g}_t)$ is 
an unbiased estimator of $\nabla_{\g}J(\theta_t)$ and the standard 
$\mathcal{O}(1/\sqrt{T})$ rate holds regardless of Fisher geometry.
The alignment parameter $\alpha$ enters through the convergence 
\emph{interpretation}: it quantifies how closely each LPG step tracks the 
natural gradient direction, giving a practitioner diagnostic for when 
projection suffices versus when explicit Fisher inversion is warranted.
The following proposition applies the Kantorovich inequality to make this precise.

\begin{proposition}[Kantorovich alignment diagnostic]
\label{prop:kappa_convergence}
Suppose 
Assumptions~(A1)--(A3),
Assumptions~(A5)--(A6), 
and~Assumption~\ref{ass:alignment} 
hold with alignment parameter $\alpha \in (0,1]$ and Fisher condition 
number $\kappa$, where $\kappa = \lambda_{\max}(F)/\lambda_{\min}(F)$.

\begin{enumerate}[label=(\roman*), leftmargin=2em]
\item \textbf{Euclidean convergence (unchanged).}
The projected gradient method satisfies the same rate as 
Corollary~\ref{cor:convergence_lpg}:
\begin{equation}
\label{eq:euclidean_rate}
\min_{0 \le t \le T-1} \E\!\left[\|\nabla_{\g} J(\theta_t)\|_F^2\right]
\le 
\frac{2(J^* - J(\theta_0))}{\eta\sqrt{T}}
+ \frac{L\eta\sigma_g^2}{\sqrt{T}},
\end{equation}
where $\eta$ is the learning rate scale.

\item \textbf{Natural gradient stationarity conversion.}
At any iterate $\theta_t$, the natural gradient norm and the Euclidean 
gradient norm are related by
\begin{equation}
\label{eq:nat_grad_conversion}
\frac{1}{\lambda_{\max}(F)} \|\nabla_{\g} J\|_F^2 
\le \langle \nabla J, F^{-1}\nabla J \rangle 
\le \frac{1}{\lambda_{\min}(F)} \|\nabla_{\g} J\|_F^2.
\end{equation}
Thus Euclidean $\varepsilon$-stationarity 
$(\|\nabla_{\g} J\|_F^2 \le \varepsilon)$ implies natural gradient 
$(\varepsilon/\lambda_{\min})$-stationarity: the conversion cost is a 
factor $1/\lambda_{\min}(F)$, not a bias floor.

\item \textbf{Directional quality of each step.}
At each iterate, the LPG update direction $d_t = \nabla_{\g} J(\theta_t)$ 
satisfies
\begin{equation}
\label{eq:directional_quality}
\cos(d_t,\, F(\theta_t)^{-1}\nabla J(\theta_t)) 
\ge \alpha 
\ge \frac{2\sqrt{\kappa}}{\kappa + 1},
\end{equation}
so the Euclidean step is within angle $\arccos(\alpha)$ of the natural gradient direction.
\end{enumerate}
\end{proposition}

\begin{proof}
[Proof sketch]
Part~(i): For Lie-algebraic policies, the policy gradient already lives 
in $\g$ without projection---a consequence of the score function 
expansion and the fact that actions are $\g$-valued 
(Supplement~\S\ref{sec:appendix_convergence_proofs}).
Part~(ii): Expand $P_{\g}(\nabla J)$ in the eigenbasis of $F(\theta)$ and 
apply the Cauchy--Schwarz-type estimate 
$\langle u, F^{-1/2}v\rangle^2 \le \lambda_{\max}/\lambda_{\min} \cdot \|u\|^2\|v\|^2$; 
the $\kappa$-dependent bounds follow directly.
Part~(iii): Follows from Assumption~\ref{ass:alignment} verified by the Kantorovich 
inequality in Supplement~\S\ref{sec:appendix_alignment}.
\end{proof}

\begin{remark}[Decision rule]
\label{rem:kantorovich_classical}
The bound $\alpha \ge 2\sqrt{\kappa}/(\kappa+1)$ follows from the Kantorovich 
inequality (Supplement~\S\ref{sec:kantorovich}). 
If $\kappa < 5$ (bound $\ge 0.745$): use LPG.
If $\kappa > 10$ (bound $\le 0.575$): prefer Fisher inversion.
Our $\SO(3)^J$ experiments give $\kappa \approx 2.5$ (empirical $\alpha = 0.971$), 
firmly in the LPG-recommended regime.
\end{remark}

\begin{table}[t]
\centering
\caption{Convergence by algebra type and Fisher isotropy. Anisotropy 
degrades the \emph{interpretation} of stationarity (by factor $\kappa$) 
but not the algorithmic rate.}
\label{tab:convergence_classification}
\scriptsize
\begin{tabular}{@{}lcc@{}}
\toprule
& \textbf{Isotropic} ($\kappa \approx 1$) 
& \textbf{Anisotropic} ($\kappa \gg 1$) \\
\midrule
\textbf{Compact} 
& $\mathcal{O}(1/\sqrt{T})$, no $R$-dep.\ 
& $\mathcal{O}(1/\sqrt{T})$; nat.\ grad.\ gap $\times \kappa$ \\
\textbf{Non-compact} 
& $\mathcal{O}(e^{2R}/\sqrt{T})$ 
& $\mathcal{O}(e^{2R}/\sqrt{T})$; nat.\ grad.\ gap $\times \kappa$ \\
\bottomrule
\end{tabular}
\end{table}

\begin{corollary}[Compact vs.\ non-compact convergence]
\label{cor:compact_convergence}
Under Assumptions~(A1)--(A6) and Assumption~\ref{ass:alignment}:
\begin{enumerate}[label=(\roman*),leftmargin=2em]
\item \textbf{Compact} $\g \subseteq \uu(n)$: $L = \mathcal{O}(1)$ independent of $R_0$
(Theorem~\ref{thm:dichotomy}), giving $\min_{t\le T}\E[\|P_{\g}(\nabla J)\|_F^2] = \mathcal{O}(1/\sqrt{T})$.
Radius projection triggers on $<2\%$ of iterations (Table~\ref{tab:scalability}),
so Assumption~(A3) is non-restrictive in practice.
\item \textbf{Non-compact} $\g$ with hyperbolic element: $L = \Omega(e^{2R_0})$
(Theorem~\ref{thm:dichotomy}), degrading the rate to $\mathcal{O}(e^{2R_0}/\sqrt{T})$;
anisotropy ($\kappa>1$) further compounds via the natural gradient conversion factor
(Proposition~\ref{prop:kappa_convergence}(ii)).
\end{enumerate}
\end{corollary}

\section{Algorithm and Computational Complexity}
\label{sec:algorithms}

Projection onto classical matrix Lie algebras admits closed-form formulas
at $\mathcal{O}(n^2)$ per block; Table~\ref{tab:complexity} compares per-iteration
costs against exact and approximate natural gradient methods.

\subsection{Algorithm}

\begin{algorithm}[t]
\caption{Lie--Projected Policy Gradient (LPG)}
\label{alg:lpg}
\begin{algorithmic}[1]
\STATE \textbf{Input:} initial parameter $\theta_0 \in \g$, stepsizes 
$\{\eta_t\}$, parameter bound $B_\theta > 0$, iterations $T$
\FOR{$t = 0,\dots,T-1$}
    \STATE Collect trajectories and compute stochastic gradient $g_t$:
    \[
        \E[g_t \mid \theta_t] = \nabla J(\theta_t).
    \]
    \STATE Project onto Lie algebra:
    $v_t = P_{\g}(g_t)$.
    \STATE Gradient step:
    $\tilde{\theta}_{t+1} = \theta_t + \eta_t v_t$.
    \STATE Radius projection (ensures Assumption~\ref{ass:bounded_iterates}):
    \[
        \theta_{t+1} = 
        \begin{cases}
        \tilde{\theta}_{t+1} & \text{if } 
            \|\tilde{\theta}_{t+1}\|_F \le B_\theta, \\[3pt]
        B_\theta \cdot \tilde{\theta}_{t+1} / \|\tilde{\theta}_{t+1}\|_F 
            & \text{otherwise}.
        \end{cases}
    \]
\ENDFOR
\STATE \textbf{Output:} $\theta_T$ or uniform average 
$\bar{\theta}_T = T^{-1}\sum_{t=0}^{T-1} \theta_t$.
\end{algorithmic}
\end{algorithm}

\begin{remark}[No retraction needed]
\label{rem:radius_projection}
Since $\g$ is a linear subspace, $\theta_t + \eta_t v_t \in \g$ whenever
both summands lie in $\g$: no retraction or exponential map is required
in the optimization step (the exponential map appears only in the
policy parameterization). The radius projection preserves $\g$ by rescaling;
for compact algebras the trigger rate is under $2\%$
(Table~\ref{tab:scalability}, $J \le 30$ runs).
\end{remark}

The projection step is the only operation not present in classical policy
gradient, and for all classical algebras---notably $\so(3)^J$---this projection
reduces to blockwise skew-symmetrization, costing $\mathcal{O}(n^2 J)$ flops~\cite{lezcano2019trivializations}.

\subsection{Complexity comparison}

Let $d_{\g} = \dim(\g)$ denote the intrinsic Lie-algebra dimension and $n$ the
matrix size (e.g., $n=3$ for $\SO(3)$). Table~\ref{tab:complexity} summarizes
the per-iteration complexities of widely used policy optimization methods.

\begin{table}[t]
\centering
\caption{Per-iteration complexity.}
\label{tab:complexity}
\scriptsize
\setlength{\tabcolsep}{4pt}
\begin{tabular}{@{}lcc@{}}
\toprule
\textbf{Method} & \textbf{Cost} & \textbf{Dom.\ op.} \\
\midrule
Euclid.\ PG & $\mathcal{O}(d_{\g})$ & grad.\ acc. \\
Exact NG & $\mathcal{O}(d_{\g}^3)$ & Fisher inv. \\
CG ($k$ it.) & $\mathcal{O}(k d_{\g}^2)$ & Fisher-vec. \\
K-FAC & $\mathcal{O}(d_{\g}^2)$ & Kronecker \\
\midrule
\textbf{LPG} & $\mathcal{O}(n^2 J)$ & proj.\ $\g$ \\
\bottomrule
\end{tabular}
\end{table}

For $\SO(3)^J$, $d_{\g} = 3J$ and $n = 3$, so
$\mathcal{O}(n^2 J) = \mathcal{O}(9J)$---linear in joints versus cubic 
$\mathcal{O}((3J)^3)$ for exact natural gradient.
For $J=10$, projection costs ${\sim}90$ flops vs.\ ${\sim}27{,}000$ for Fisher 
inversion ($1.1$--$1.7\times$ wall-clock reduction; see Section~\ref{sec:experiments}).

\section{Numerical Illustrations}
\label{sec:experiments}

Each subsection validates a specific theoretical prediction from
Section~\ref{sec:geometric_inputs} and Section~\ref{sec:convergence}; Table~\ref{tab:theory_experiment}
maps claims to experiments. All experiments use controlled $\SO(3)^J$
motion-control tasks with exact Lie-algebraic structure, enabling clean
isolation from confounds such as partial observability or contact dynamics.
Results are averaged over five random seeds; error bars show one standard deviation.
The central data-science takeaway is that Lie-algebraic parameterization 
improves both sample efficiency (fewer interactions to reach a given return) 
and per-step computational cost relative to ambient and natural-gradient 
baselines, with the magnitude of improvement predicted by the algebra type 
via Theorem~\ref{thm:dichotomy}.
Section~\ref{sec:se3_experiment} tests the non-compact setting, Section~\ref{sec:robustness}
tests robustness under symmetry violations (with extended results in the Supplement), and Section~\ref{sec:method_comparison}
benchmarks LPG against baselines.

\begin{table}[t]
\centering
\caption{Theory--experiment mapping. Each entry validates a specific 
theoretical claim; the source column gives the theorem or proposition 
establishing the prediction.}
\label{tab:theory_experiment}
\scriptsize
\setlength{\tabcolsep}{4pt}
\begin{tabular}{@{}p{3.2cm}lc@{}}
\toprule
\textbf{Claim} & \textbf{Source} & \textbf{Expt.} \\
\midrule
Fisher isotropy \& alignment & Prop.~\ref{prop:block_fisher}, \ref{prop:kappa_convergence} & \S\ref{sec:fisher_alignment} \\
Anisotropy degradation & Prop.~\ref{prop:kappa_convergence} & \S\ref{sec:fisher_alignment} \\
$d_{\g}$-sample scaling & Prop.~\ref{prop:dimension_reduction} & \S\ref{sec:fisher_alignment} \\
Convergence regimes & Thm.~\ref{thm:dichotomy} & \S\ref{sec:convergence_rate_experiment} \\
Computational scaling & \S\ref{sec:algorithms} & \S\ref{sec:scalability_ablation} \\
Compact vs.\ non-compact & Cor.~\ref{cor:compact_convergence} & \S\ref{sec:se3_experiment} \\
Robustness to sym.\ violations & Remark~\ref{rem:operating_regime} & \S\ref{sec:robustness} \\
LPG vs.\ baselines & Prop.~\ref{prop:kappa_convergence} & \S\ref{sec:method_comparison} \\
\bottomrule
\end{tabular}
\end{table}

\subsection{Experimental Setup}

\paragraph{Environment.}
Synthetic $\SO(3)^{J}$ system with $J=10$ rotational joints; reward is 
the negative squared geodesic error 
$r(s,a) = -\sum_j \|\log(R_j^\top R_j^{\mathrm{target}})\|_F^2$.

\paragraph{Policies.}
We compare two parameterizations:
(a)~\textbf{Lie policy}: $\mu_\theta(s) \in \so(3)^J$,
$d_{\g} = 30$ parameters, $\sigma = 0.1$;
(b)~\textbf{Ambient policy}: $\mu_\theta(s) \in \R^{3J}$ via random 
projection from $\R^{45J}$ ($15\times$ over-parameterization, $d = 450$).
Actions are clipped to $\|a_t - \mu_\theta(s_t)\|_F \le 3\sigma$ 
(Assumption~\ref{ass:bounded_actions}). States update via 
$R_{j,t+1} = R_{j,t}\exp(\omega_{j,t})$ per joint.
Training uses REINFORCE with advantage normalization, 
$\eta = 0.25$, eight episodes per iteration, forty iterations.

\paragraph{$G$-equivariance of the environment.}
The dynamics $R_{j,t+1} = R_{j,t}\exp(\omega_{j,t})$ satisfy
$G$-equivariance (required by Proposition~\ref{prop:lossless}): left multiplication
by $g_j$ commutes with the right-acting exponential update, the
isotropic noise $\omega_j \sim \mathcal{N}(0,\sigma^2 I_3)$ is
$\mathrm{Ad}$-invariant, and the geodesic reward
$\|\log(R_j^\top R_j^*)\|_F$ is invariant under $R_j \mapsto g_j R_j$.

\subsection{Fisher Geometry, Alignment, and Sample Efficiency}
\label{sec:fisher_alignment}
\label{sec:anisotropy_ablation}
\label{sec:sample_efficiency}

The alignment between the natural gradient $F^{-1}\nabla J$ and the vanilla
gradient $\nabla J$ is measured using cosine similarity:
\[
\operatorname{align}
=
\frac{\langle F^{-1}\nabla J,\nabla J\rangle}
     {\|F^{-1}\nabla J\|\;\|\nabla J\|}.
\]
The Fisher matrix is estimated from empirical score covariances.
The \emph{isotropy deviation} is defined as 
$\varepsilon_F := \|F - \bar{\lambda}I\|_F / \|F\|_F$, 
where $\bar{\lambda} = \mathrm{tr}(F)/d_{\g}$ is the mean eigenvalue; 
$\varepsilon_F = 0$ corresponds to exact isotropy ($F \propto I$).

\paragraph{Results.}
Across $200$ measurements (five seeds $\times$ forty iterations),
Table~\ref{tab:fisher_alignment} reports alignment statistics and
Table~\ref{tab:controlled_anisotropy} shows results under controlled violations.

\begin{table}[t]
\centering
\scriptsize
\begin{minipage}[t]{0.33\textwidth}
\centering
\caption{Fisher--metric alignment statistics.}
\label{tab:fisher_alignment}
\setlength{\tabcolsep}{2pt}
\begin{tabular}{@{}lc@{}}
\toprule
Statistic & Value \\ \midrule
Mean alignment & $0.971$ \\
95\% CI & $[0.970, 0.972]$ \\
Std.\ dev. & $0.007$ \\
Range & $[0.945, 0.990]$ \\
$\kappa$ & $2.53 \pm 0.12$ \\
$\varepsilon_F$ & $0.24 \pm 0.01$ \\
\bottomrule
\end{tabular}
\end{minipage}\hfill
\begin{minipage}[t]{0.63\textwidth}
\centering
\caption{Alignment under controlled isotropy violations.}
\label{tab:controlled_anisotropy}
\setlength{\tabcolsep}{2pt}
\begin{tabular}{@{}lcccc@{}}
\toprule
Condition & $\kappa$ & $\varepsilon_F$ & Alignment & Final Return \\
\midrule
Uniform (baseline) & $2.53$ & $0.24$ & $0.971$ & $-695.0$ \\
Axis-biased & $3.08$ & $0.29$ & $0.954$ & $-700.5$ \\
Correlated ($\kappa_M{=}5$) & $6.83$ & $0.49$ & $0.884$ & $-754.3$ \\
Correlated ($\kappa_M{=}10$) & $12.98$ & $0.64$ & $0.816$ & $-799.2$ \\
\bottomrule
\end{tabular}
\end{minipage}
\end{table}

\paragraph{Comparison with theoretical bounds.}
The $\kappa$-based bound gives $\cos \ge 2\sqrt{2.53}/(2.53 + 1) = 0.901$; 
the empirical mean ($0.971$) exceeds this by $7.8\%$, confirming that 
worst-case configurations rarely occur.

\begin{figure}[t]
\centering
\begin{minipage}[t]{0.48\textwidth}
\centering
\IfFileExists{figs/exp1_fisher_alignment_hist.png}{\includegraphics[width=\textwidth]{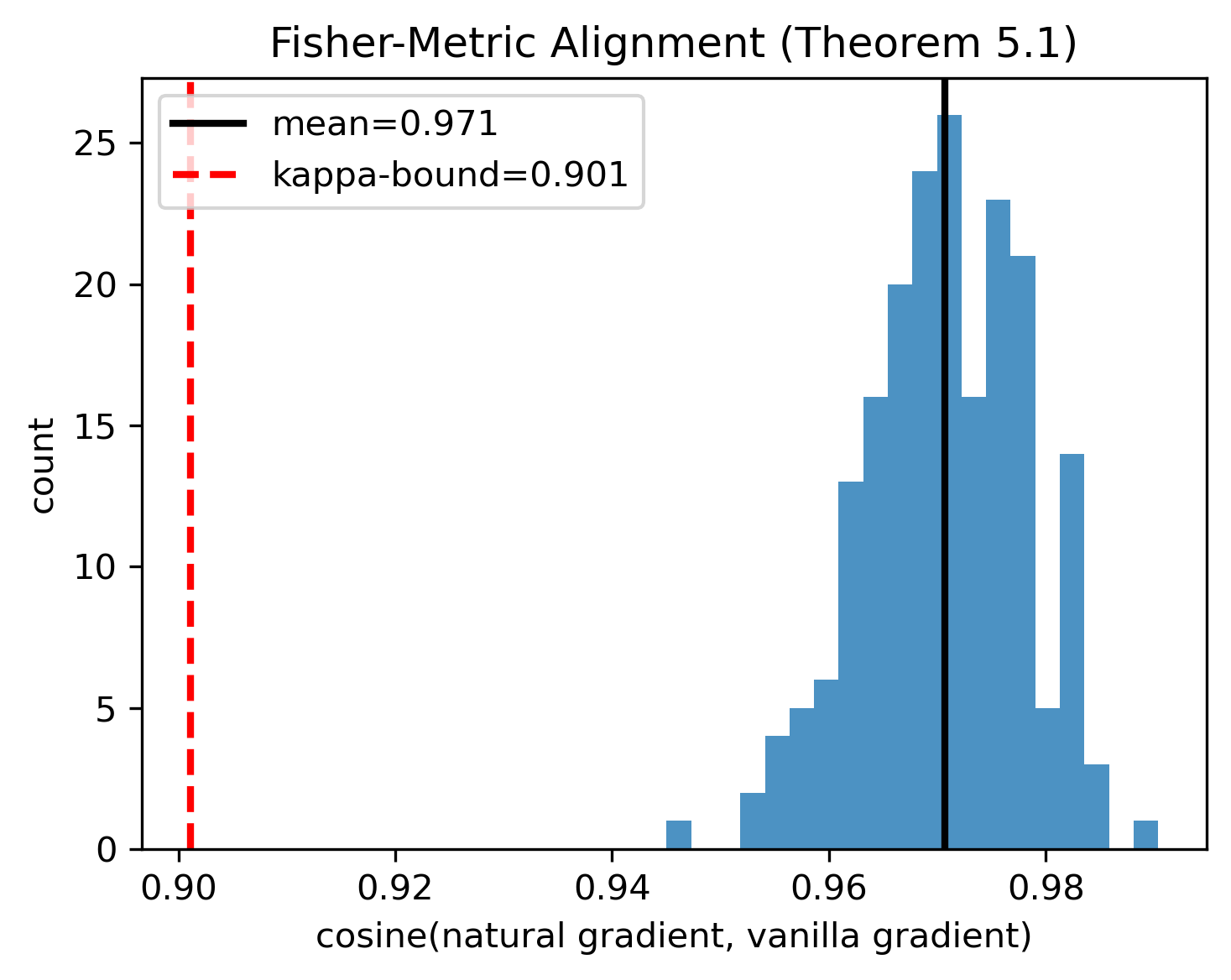}}{\fbox{\parbox[c][3cm][c]{6cm}{\centering\small[exp1_fisher_alignment_hist.png]}}}
\end{minipage}\hfill
\begin{minipage}[t]{0.48\textwidth}
\centering
\IfFileExists{figs/isotropy_tracking.png}{\includegraphics[width=\textwidth]{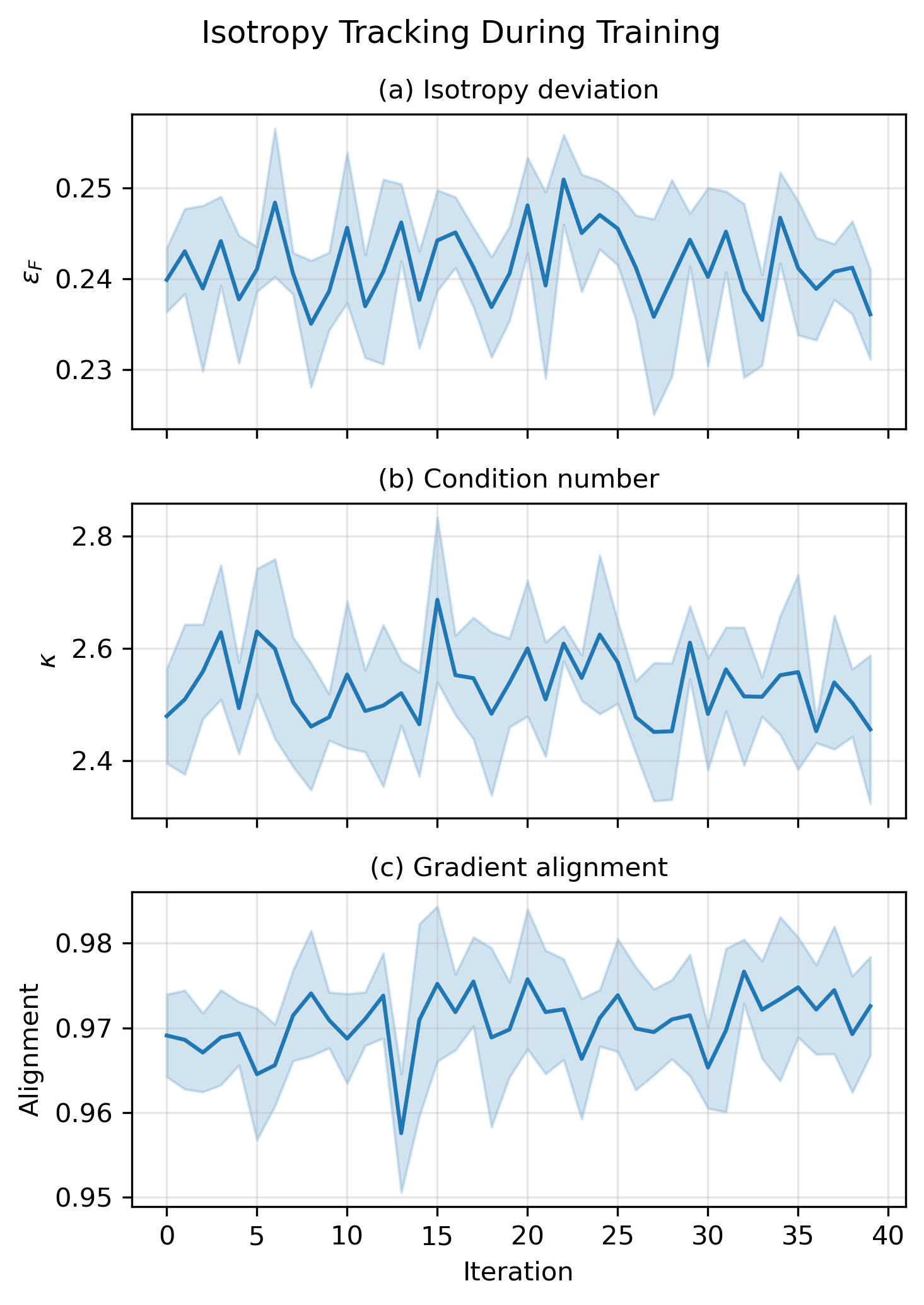}}{\fbox{\parbox[c][3cm][c]{6cm}{\centering\small[isotropy_tracking.png]}}}
\end{minipage}
\caption{\textbf{Left:} Fisher--metric alignment histogram (200 measurements); 
vertical lines show empirical mean (solid) and $\kappa$-based bound (dashed).
\textbf{Right:} Isotropy metrics during training ($J=10$, 5 seeds): 
(a)~$\varepsilon_F(t)$ stays below $0.3$,
(b)~$\kappa(t) \in [1.9, 2.8]$,
(c)~alignment exceeds theoretical bound throughout.}
\label{fig:fisher_and_isotropy}
\end{figure}

\paragraph{Eigenvalue analysis.}
Eigenvalue decomposition of the empirical Fisher confirms near-isotropy: 
effective rank $r_{\mathrm{eff}} = 28.2 \pm 0.1$ (out of $d_{\g} = 30$), 
$\kappa = 2.53 \pm 0.12$, $\varepsilon_F = 0.24 \pm 0.01$. 
Tracking over 40 iterations (Figure~\ref{fig:fisher_and_isotropy}, right) shows 
no drift: $\varepsilon_F(t)$ remains stable near $0.24$; alignment 
exceeds $0.94$ throughout.

\paragraph{Approximate equivalence (Assumption~\ref{ass:alignment}).}
Across $100$ synthetic Fisher matrices with controlled anisotropy, alignment 
correlates with isotropy deviation at $r = -0.875$ ($p < 10^{-4}$), with 
empirical fit $\mathrm{align} \approx 1.01 - 0.15\,\varepsilon_F$---the small 
coefficient ($0.15$ vs.\ theoretical worst-case $4$) highlights the 
conservatism of Assumption~\ref{ass:alignment}.

\paragraph{Anisotropy ablation.}

We study alignment degradation under controlled isotropy violations.
In synthetic diagonal Fisher tests, alignment degrades smoothly from 
$1.0$ ($\kappa = 1$) to $0.651$ ($\kappa = 10$), remaining well above 
the vacuous $1 - 4\varepsilon_F$ bound.

\paragraph{Realistic isotropy violations.}
We introduce controlled per-dimension noise scaling in the 
$\SO(3)^{10}$ environment: axis-biased ($\sqrt{1.5}{\times}$ noise 
on $z$-components) and globally correlated 
($\sigma_i$ spread from $\sigma_{\mathrm{base}}$ to 
$\sigma_{\mathrm{base}}\!\sqrt{\kappa_M}$ for 
$\kappa_M \in \{5, 10\}$).

Alignment degrades monotonically with $\kappa$ 
(Figure~\ref{fig:anisotropy_and_efficiency}), but returns are also sensitive: 
at $\kappa \approx 13$, ${\sim}15\%$ return degradation. All empirical 
alignments exceed the $\kappa$-based theoretical bound. We recommend $\kappa < 10$.

\paragraph{Sample efficiency.}

The Lie parameterization uses $d_\g=30$ parameters versus $450$ ambient.
Over $400$ iterations (Figure~\ref{fig:anisotropy_and_efficiency}), the Lie policy 
achieves final return $-908.95$ versus $-1324.76$ ambient, with AUC ratio 
$0.63$ ($37\%$ improvement). The ambient policy collapses early due to 
ill-conditioning, consistent with Proposition~\ref{prop:dimension_reduction}.

\begin{figure}[t]
\centering
\begin{minipage}[t]{0.48\textwidth}
\centering
\IfFileExists{figs/controlled_anisotropy.png}{\includegraphics[width=\textwidth]{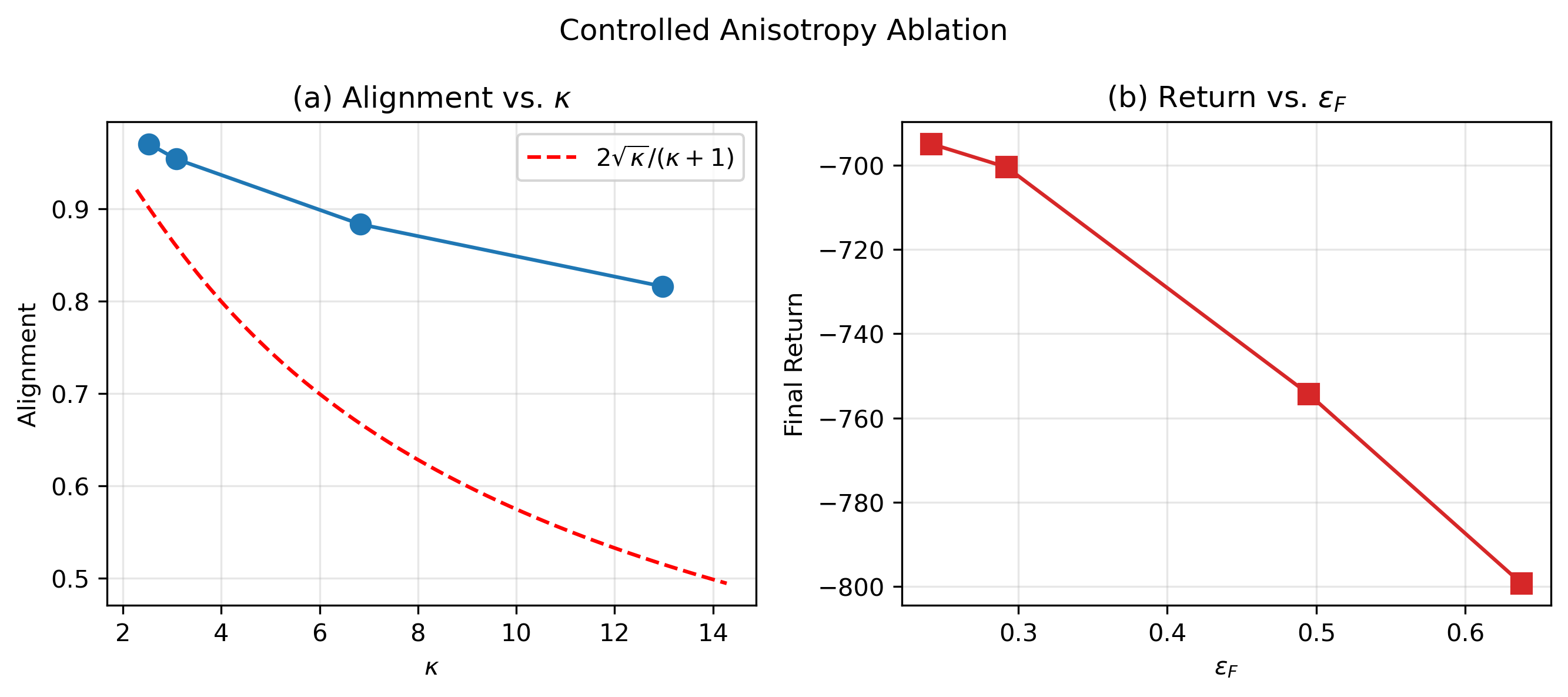}}{\fbox{\parbox[c][3cm][c]{6cm}{\centering\small[controlled_anisotropy.png]}}}
\end{minipage}\hfill
\begin{minipage}[t]{0.48\textwidth}
\centering
\IfFileExists{figs/exp3_learning_curves.png}{\includegraphics[width=\textwidth]{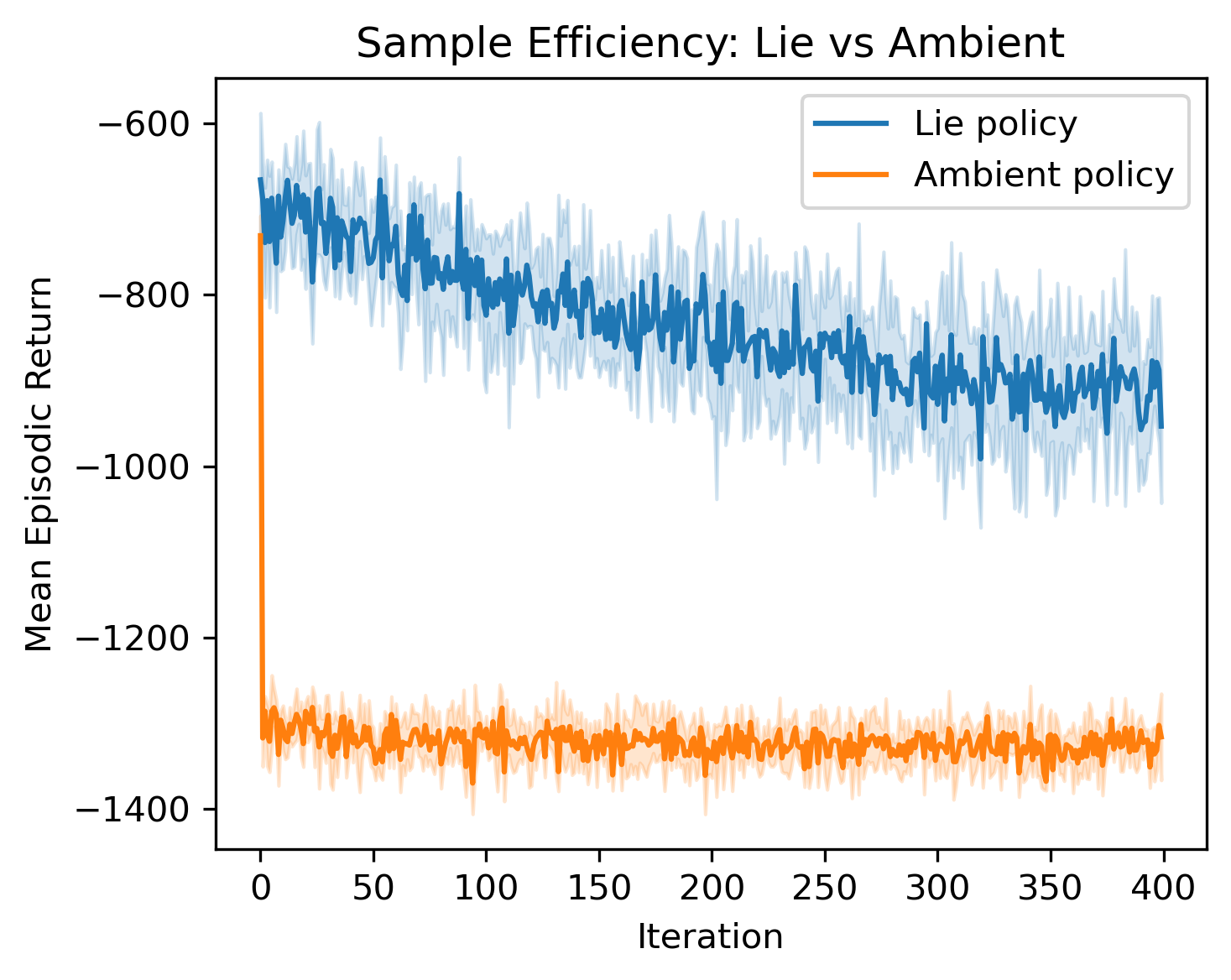}}{\fbox{\parbox[c][3cm][c]{6cm}{\centering\small[exp3_learning_curves.png]}}}
\end{minipage}
\caption{\textbf{Left:} Controlled anisotropy---alignment vs.\ $\kappa$ 
(a) and return vs.\ $\varepsilon_F$ (b).
\textbf{Right:} Sample efficiency---Lie policy (blue) vs.\ ambient 
(orange); shaded: $\pm 1$ std.}
\label{fig:anisotropy_and_efficiency}
\end{figure}

\subsection{Convergence Rate Illustration}
\label{sec:convergence_rate_experiment}

We empirically illustrate the convergence regimes predicted by the 
geometric classification of Theorem~\ref{thm:dichotomy}.

\paragraph{Controlled setting (quadratic objective).}
Projected SGD on a quadratic restricted to $\so(3)^{10}$ gives
log--log slope $-0.98$ (Figure~\ref{fig:convergence}, left), consistent with 
the $\mathcal{O}(1/T)$ deterministic rate when $L = \mathcal{O}(1)$
(Proposition~\ref{prop:compact_advantage}). The favorable exponent reflects the
$R$-independent Lipschitz constant: step sizes need not decay to
compensate for exponential growth.

\paragraph{Stochastic setting.}
On a $d_{\g}=30$ stochastic quadratic proxy (same dimension as $\so(3)^{10}$, 
Gaussian gradient noise $\sigma_g=1$, step size $\eta/\sqrt{T}$), 
the running-average log--log slope 
is $-0.52 \pm 0.08$ (five seeds, Figure~\ref{fig:convergence}, right), matching 
the $\mathcal{O}(T^{-1/2})$ stochastic rate (Corollary~\ref{cor:convergence_lpg}).
The deterministic--stochastic gap reflects the $L\eta\sigma_g^2/\sqrt{T}$
variance term dominating the convergence bound.

\begin{figure}[t]
\centering
\begin{minipage}[t]{0.48\textwidth}
\centering
\IfFileExists{figs/exp4_convergence_loglog.png}{\includegraphics[width=\textwidth]{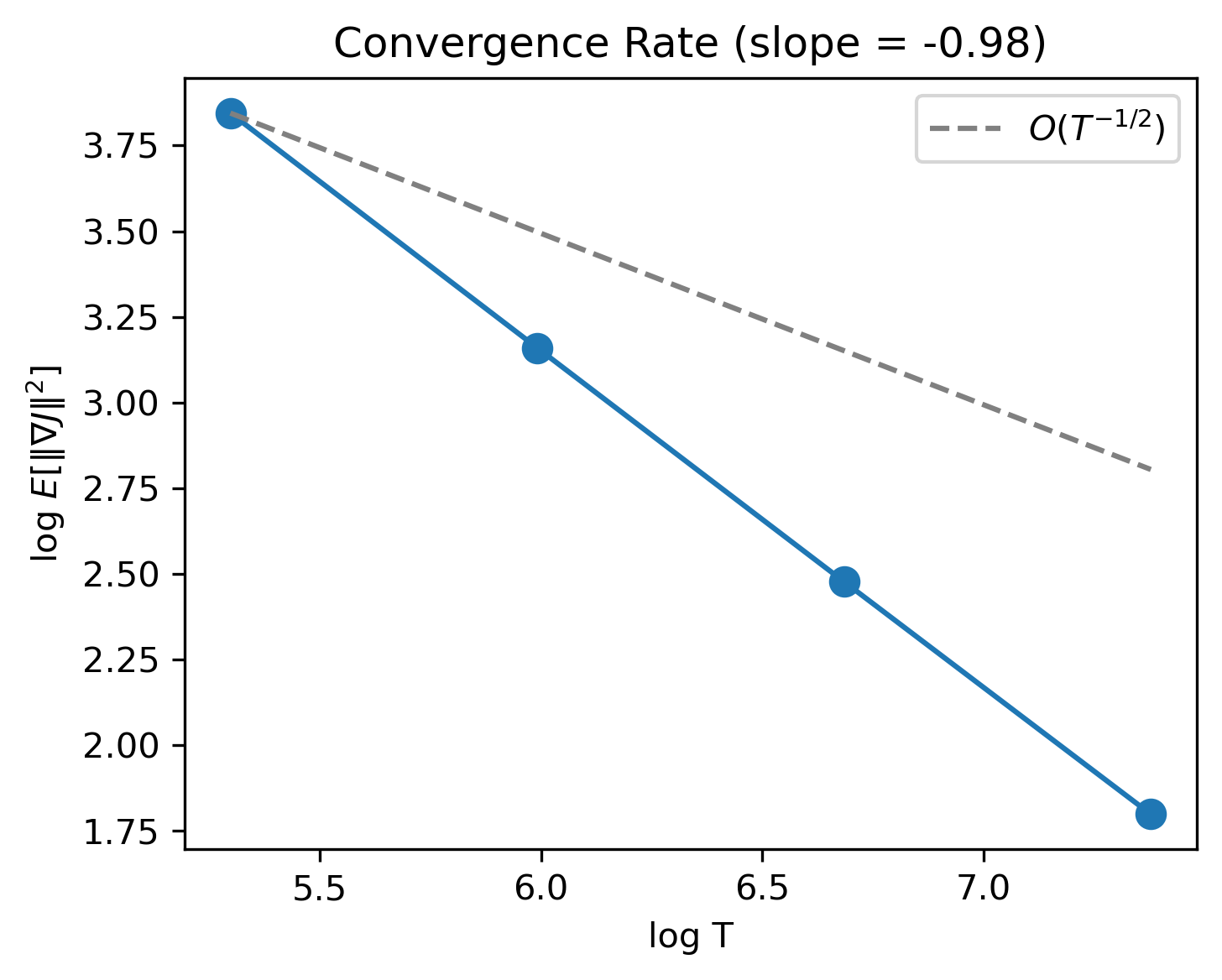}}{\fbox{\parbox[c][3cm][c]{5.5cm}{\centering\small[exp4_convergence_loglog.png]}}}
\end{minipage}\hfill
\begin{minipage}[t]{0.48\textwidth}
\centering
\IfFileExists{figs/exp4_convergence_loglog_stochastic.png}{\includegraphics[width=\textwidth]{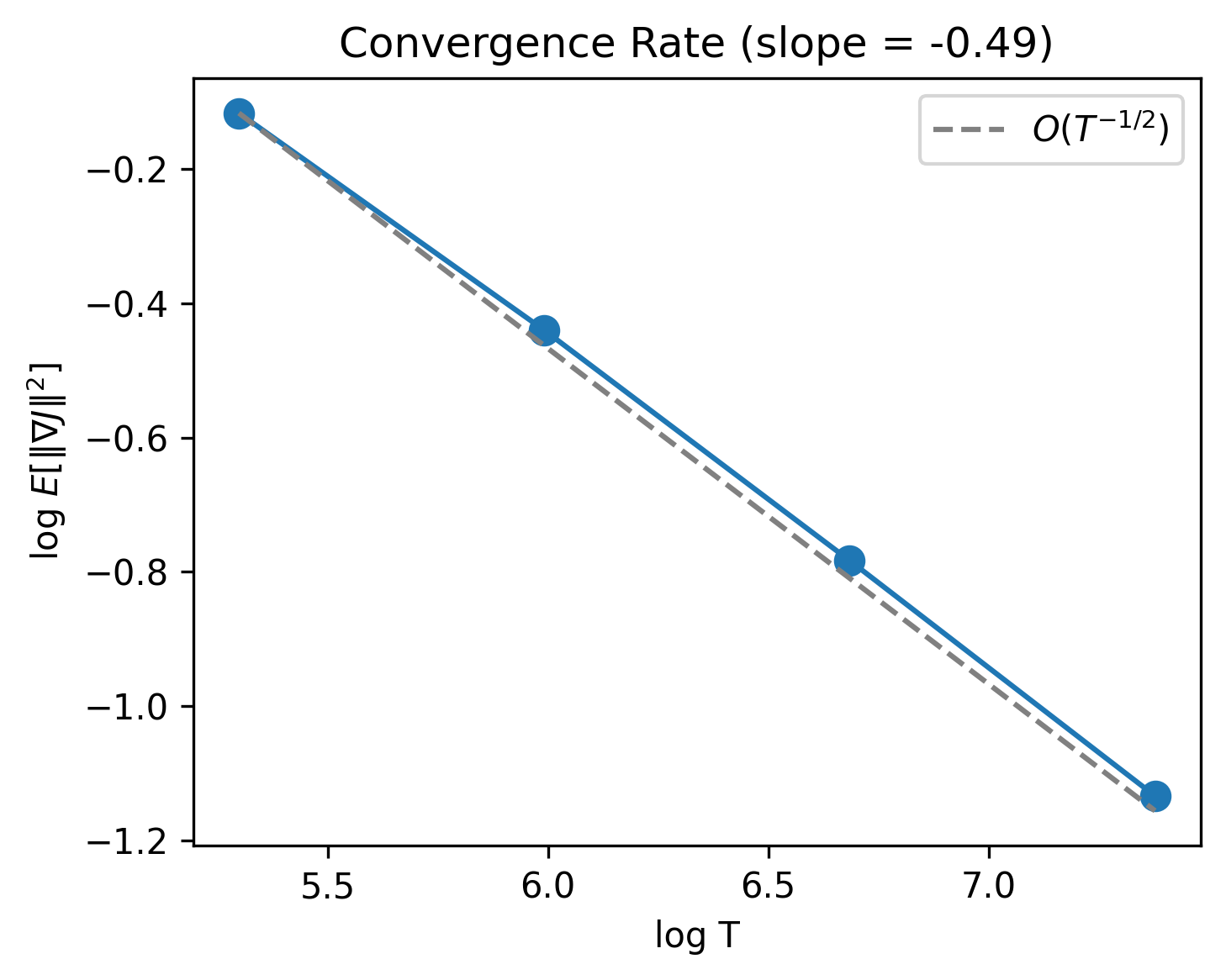}}{\fbox{\parbox[c][3cm][c]{5.5cm}{\centering\small[exp4_convergence_loglog_stochastic.png]}}}
\end{minipage}
\caption{Log--log convergence rate. \textbf{Left:} Deterministic 
(quadratic on $\so(3)^{10}$, $d_{\g}=30$), slope $-0.98$. \textbf{Right:} 
Stochastic quadratic proxy ($d_{\g}=30$, Gaussian noise $\sigma_g=1$), slope $-0.52 \pm 0.08$. 
Both consistent with compact-algebra predictions 
(Theorem~\ref{thm:dichotomy} and Corollary~\ref{cor:convergence_lpg}).}
\label{fig:convergence}
\end{figure}

\subsection{Computational Scaling and Joint-Count Ablation}
\label{sec:computational_efficiency}
\label{sec:scalability_ablation}

Section~\ref{sec:algorithms} established the $\mathcal{O}(n^2 J)$ vs.\ 
$\mathcal{O}(d_{\g}^3)$ asymptotic advantage of Lie projection over 
Fisher inversion; here we verify this empirically. We benchmark Fisher 
inversion (Cholesky) vs.\ blockwise skew-symmetrization for $\so(3)^J$ 
(mean of 200 runs per condition), with results in Table~\ref{tab:timing}.
Projection speedups range from $1.1\times$ to $1.7\times$, but the 
end-to-end advantage is modest because environment interaction and 
gradient estimation dominate wall-clock time. The projection step itself 
avoids $\mathcal{O}(d_{\g}^3)$ matrix inversion entirely.

\begin{table}[t]
\centering
\caption{Timing: Fisher inversion vs.\ Lie projection.}
\label{tab:timing}
\scriptsize
\setlength{\tabcolsep}{2pt}
\begin{tabular}{@{}ccccc@{}}
\toprule
$J$ & $d_{\g}$ & Fisher & Proj. & Speedup \\
 & & ($\mu$s) & ($\mu$s) & \\
\midrule
5 & 15 & 33.4 & 18.6 & $1.8\times$ \\
10 & 30 & 71.4 & 47.2 & $1.5\times$ \\
30 & 90 & 152.1 & 112.2 & $1.4\times$ \\
\bottomrule
\end{tabular}
\end{table}

\paragraph{Joint-count ablation.}

We vary $J \in \{5, 10, 15, 20, 30\}$ and measure optimization 
performance and Fisher isotropy (Table~\ref{tab:scalability}). For each $J$, both Lie-parameterized 
and ambient policies are trained for $200$ iterations with identical 
hyperparameters (ambient uses $15 \times d_{\g}$ parameters).

\begin{table}[t]
\centering
\caption{Scalability with number of joints $J$.}
\label{tab:scalability}
\scriptsize
\setlength{\tabcolsep}{2pt}
\begin{tabular}{@{}cccccc@{}}
\toprule
$J$ & $d_{\g}$ & Alignment & AUC Ratio & Speedup & $\varepsilon_F$ \\
\midrule
5 & 15 & $0.986 \pm 0.005$ & $0.58$ & $1.7\times$ & $0.17$ \\
10 & 30 & $0.971 \pm 0.007$ & $0.59$ & $1.3\times$ & $0.24$ \\
15 & 45 & $0.955 \pm 0.009$ & $0.61$ & $1.1\times$ & $0.29$ \\
20 & 60 & $0.939 \pm 0.011$ & $0.61$ & $1.1\times$ & $0.33$ \\
30 & 90 & $0.906 \pm 0.014$ & $0.62$ & $1.1\times$ & $0.39$ \\
\bottomrule
\end{tabular}
\end{table}

Fisher alignment remains above $0.90$ even at $J = 30$, confirming 
graceful degradation. The Lie policy consistently outperforms ambient 
at all scales (AUC ratio $0.58$--$0.62$). Isotropy deviation grows 
roughly as $\sqrt{d_{\g}}$, consistent with Assumption~\ref{ass:alignment}.

At larger scales ($J \in \{50,100,200\}$), Fisher inversion grows cubically 
while Lie projection grows linearly; projection speedup exceeds $100\times$ 
at $J=200$ with alignment remaining above $0.86$ ($\kappa \approx 4.1$). 
Full benchmarks in Supplement~\S\ref{sec:supp_scaling} (Table~S5).

\subsection{Symmetry-Violation Robustness}
\label{sec:robustness}

Under controlled $G$-equivariance-breaking perturbations to $\SO(3)^{10}$ 
(stochastic transitions $\sigma_\epsilon \in \{0.01,0.05,0.1\}$, observation noise, 
and reward noise), alignment remains above $0.94$ and $\kappa$ stays below $3.2$, 
demonstrating graceful rather than catastrophic degradation. Full results appear 
in Supplement~\S\ref{sec:supp_robustness} (Table~S6).

\subsection{Non-Compact Algebra: \texorpdfstring{$\mathrm{SE}(3)$}{SE(3)} Rigid-Body Control}
\label{sec:se3_experiment}

We test the non-compact setting on $\SE(3)$, whose Lie algebra 
$\se(3) \subset \R^{4\times 4}$ ($\dim = 6$) has an unbounded translation 
component. Elements of $\se(3)$ take the form $\bigl[\begin{smallmatrix}\Omega & v \\ 0 & 0\end{smallmatrix}\bigr]$ 
with $\Omega \in \so(3)$ skew-symmetric; eigenvalues are purely imaginary 
or zero, so $\se(3)$ contains no hyperbolic elements and the 
$\Theta(e^{2R})$ exponential barrier of Theorem~\ref{thm:dichotomy} does not 
apply. Nevertheless, $\se(3)$ is non-compact and the translation 
component introduces polynomial Lipschitz growth ($L(R_0) = \mathcal{O}(R_0^2)$), 
making radius projection necessary in practice.

\textbf{Scope of this experiment.} The $\SE(3)$ results validate two 
predictions: (i)~polynomial $L$-growth and the necessity of radius 
projection (Theorem~\ref{thm:dichotomy} and Assumption~\ref{ass:smoothness_exp}(ii)), and 
(ii)~the Kantorovich alignment bound (Proposition~\ref{prop:kappa_convergence}(iii)). 
They do \emph{not} validate Assumption~\ref{ass:concentrability} for $\SE(3)$, 
which remains an open condition for non-compact groups 
(Section~\ref{sec:prelim}); the experiment is conducted under the working 
assumption that concentrability holds for the translation-bounded 
trajectories induced by $B_\theta = 2.0$. 
Verifying this assumption would require bounding the mixing time of 
the induced Markov chain on the translation component of $\SE(3)$ 
under the radius-projected dynamics---a problem that reduces to 
ergodicity of a bounded random walk on $\R^3$ and is left as a 
concrete open direction. 
Results are summarised in Table~\ref{tab:se3_results}.

\begin{table}[t]
\centering
\scriptsize
\caption{Compact vs.\ non-compact Lie algebras.}
\label{tab:se3_results}
\setlength{\tabcolsep}{2pt}
\begin{tabular}{@{}lcccc@{}}
\toprule
Algebra & $d_{\g}$ & Alignment & $\kappa$ & $\varepsilon_F$ \\
\midrule
$\so(3)$ (compact) & 3 & $0.996 \pm 0.004$ & $1.30 \pm 0.13$ & $0.11$ \\
$\se(3)$ (non-compact) & 6 & $0.949 \pm 0.024$ & $2.79 \pm 0.38$ & $0.36$ \\
\bottomrule
\end{tabular}
\end{table}

$\SE(3)$ exhibits higher $\kappa$ ($2.79$ vs.\ $1.30$) and $\varepsilon_F$ 
($0.36$ vs.\ $0.11$) than $\SO(3)$, but alignment still exceeds $0.94$. 
Anisotropy concentrates in the translation directions (the abelian 
component, where polynomial Fr\'{e}chet bounds apply). Without radius 
projection, optimization diverges; with $B_\theta = 2.0$, training is 
stable. The alignment degradation for $\se(3)$ is modest 
($0.949$ vs.\ $0.996$)---at $\kappa \approx 2.8$ the Kantorovich bound 
gives $\alpha \ge 0.882$, and the empirical value exceeds this by $7\%$, 
about the same margin as in the compact case.

\paragraph{Reconciling theory with experiment.}
Although $\se(3)$ has no hyperbolic elements (so the $\Theta(e^{2R})$ 
barrier does not apply), training still diverges without radius projection because 
polynomial Lipschitz growth $L(R)=\mathcal{O}(R^2)$ is also unbounded: 
without projection $\|\theta\|_F$ drifts to ${\sim}18$, giving 
${\sim}324\times$ growth in $L$ relative to $R=1$ (Figure~\ref{fig:se3_divergence}). 
In short: non-compactness governs \emph{step aggressiveness}, not 
\emph{step direction quality}---radius projection is essential for both 
exponential and polynomial growth regimes.

\begin{figure}[t]
\centering
\IfFileExists{figs/exp_se3_divergence.png}{\includegraphics[width=0.85\textwidth]{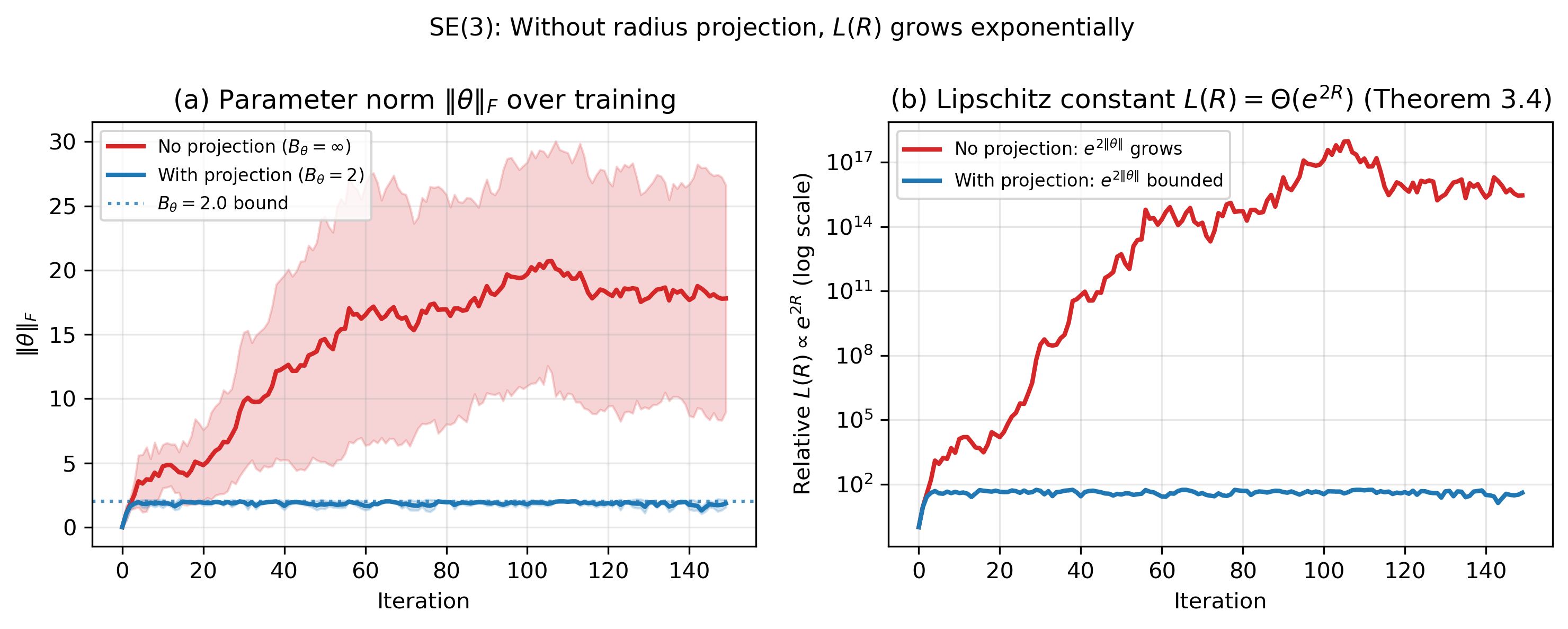}}{\fbox{\parbox[c][3cm][c]{6cm}{\centering\small[exp_se3_divergence.png]}}}
\caption{Effect of radius projection on $\SE(3)$ (5 seeds, shaded $\pm 1\sigma$).
\textbf{Left:} Without projection ($B_\theta = \infty$), the parameter norm
$\|\theta\|_F$ grows unboundedly (reaching ${\sim}18$);
with projection ($B_\theta = 2$), it remains bounded at ${\sim}1.85$.
\textbf{Right:} The theoretical gradient Lipschitz constant
$L(R) = \mathcal{O}(R^2)$ for $\se(3)$ (polynomial, not exponential, 
since $\se(3)$ has no hyperbolic elements), computed from the observed
$\|\theta\|_F$ trajectories.
Without projection, $L$ grows roughly $18^2/1^2 = 324$-fold
over training; with projection, $L$ stays within a constant factor of its
initial value.}
\label{fig:se3_divergence}
\end{figure}

\subsection{Method Comparison}
\label{sec:method_comparison}

We compare three methods on $\SO(3)^{10}$ ($\eta = 0.25$, 8 episodes/iter, 200 iters, 5 seeds):
\textbf{LPG} ($d_{\g}=30$, blockwise skew-symmetrization, REINFORCE~\cite{sutton2018reinforcement});
\textbf{Ambient PG} ($3d_{\g}=90$ params, REINFORCE; note this uses $3\times$ over-parameterization
for a fair parameter-count comparison, distinct from the $15\times$ regime in Section~\ref{sec:sample_efficiency});
\textbf{Natural gradient} ($d_{\g}=30$, Monte Carlo Fisher + CG~\cite{peters2008natural,amari1998natural}, $k=10$ iters).
CG is used in place of Cholesky here for scalability to the large-$J$ regime; 
the $1.1$--$1.7\times$ wall-clock speedup figures cited in the abstract 
and Table~\ref{tab:timing} are from the direct Cholesky benchmark in 
Section~\ref{sec:computational_efficiency}, where both methods are compared on equal footing.
The fair wall-clock comparison is the $1.1$--$1.7\times$ speedup at $J\le30$ 
(Table~\ref{tab:timing}), where both methods run identical gradient estimation 
and differ only in the projection/inversion step.

\begin{table}[t]
\centering
\scriptsize
\caption{Method comparison on $\SO(3)^{10}$ ($J=10$, 5 seeds).}
\label{tab:method_comparison}
\setlength{\tabcolsep}{3pt}
\begin{tabular}{@{}lccc@{}}
\toprule
Method & Params & Final Return & Per-step ($\mu$s) \\
\midrule
LPG (ours) & 30 & $-964.9 \pm 39.4$ & 113 \\
Ambient PG ($3\times$) & 90 & $-1293.3 \pm 38.5$ & 6 \\
Natural gradient (CG) & 30 & $-755.4 \pm 23.7$ & 225{,}253 \\
\bottomrule
\end{tabular}
\end{table}

\begin{figure}[t]
\centering
\begin{minipage}[t]{0.48\textwidth}
\centering
\IfFileExists{figs/exp9_method_comparison.png}{\includegraphics[width=\textwidth]{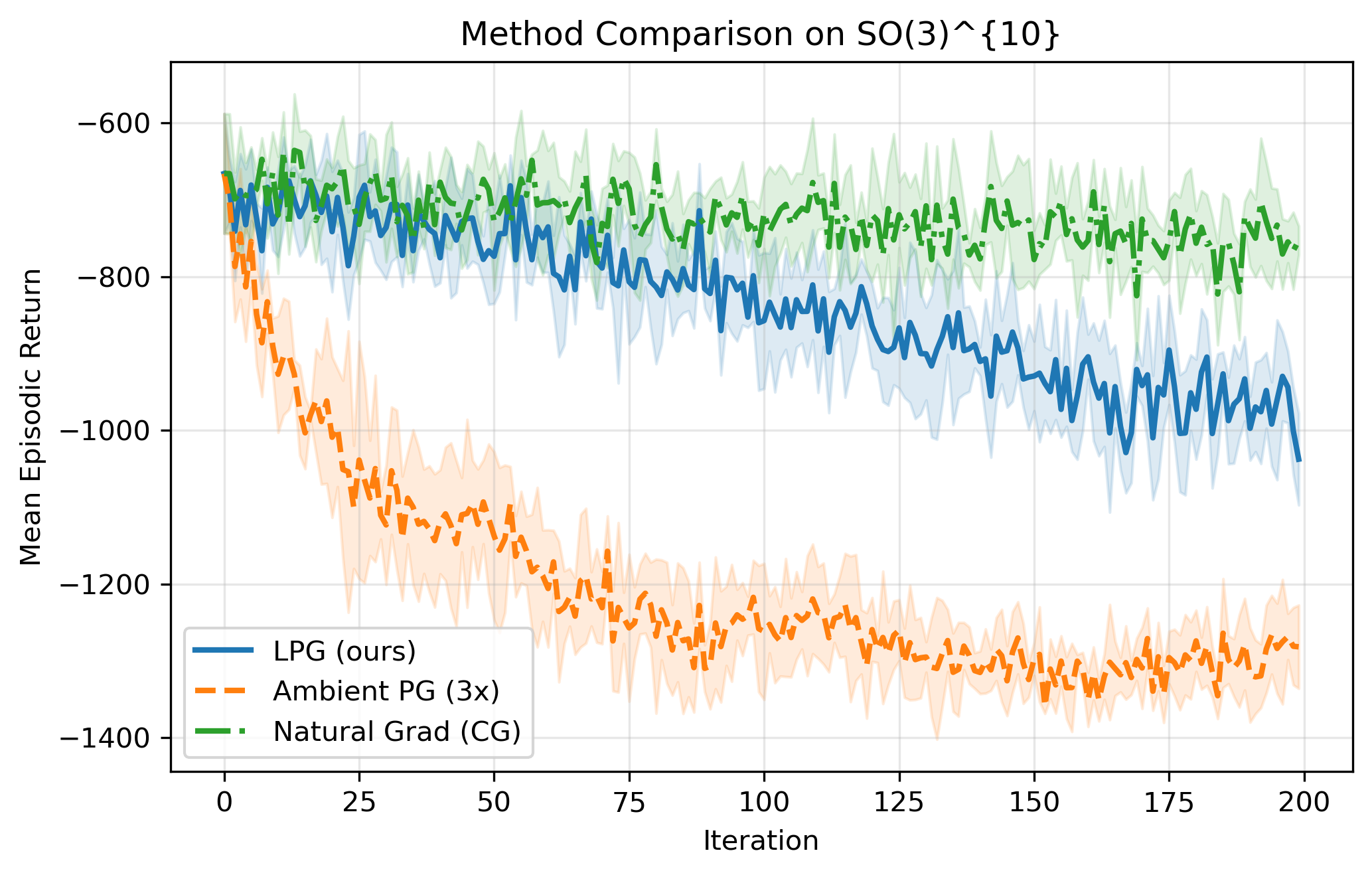}}{\fbox{\parbox[c][3cm][c]{6cm}{\centering\small[exp9_method_comparison.png]}}}
\end{minipage}\hfill
\begin{minipage}[t]{0.48\textwidth}
\centering
\IfFileExists{figs/exp9_wallclock_comparison.png}{\includegraphics[width=\textwidth]{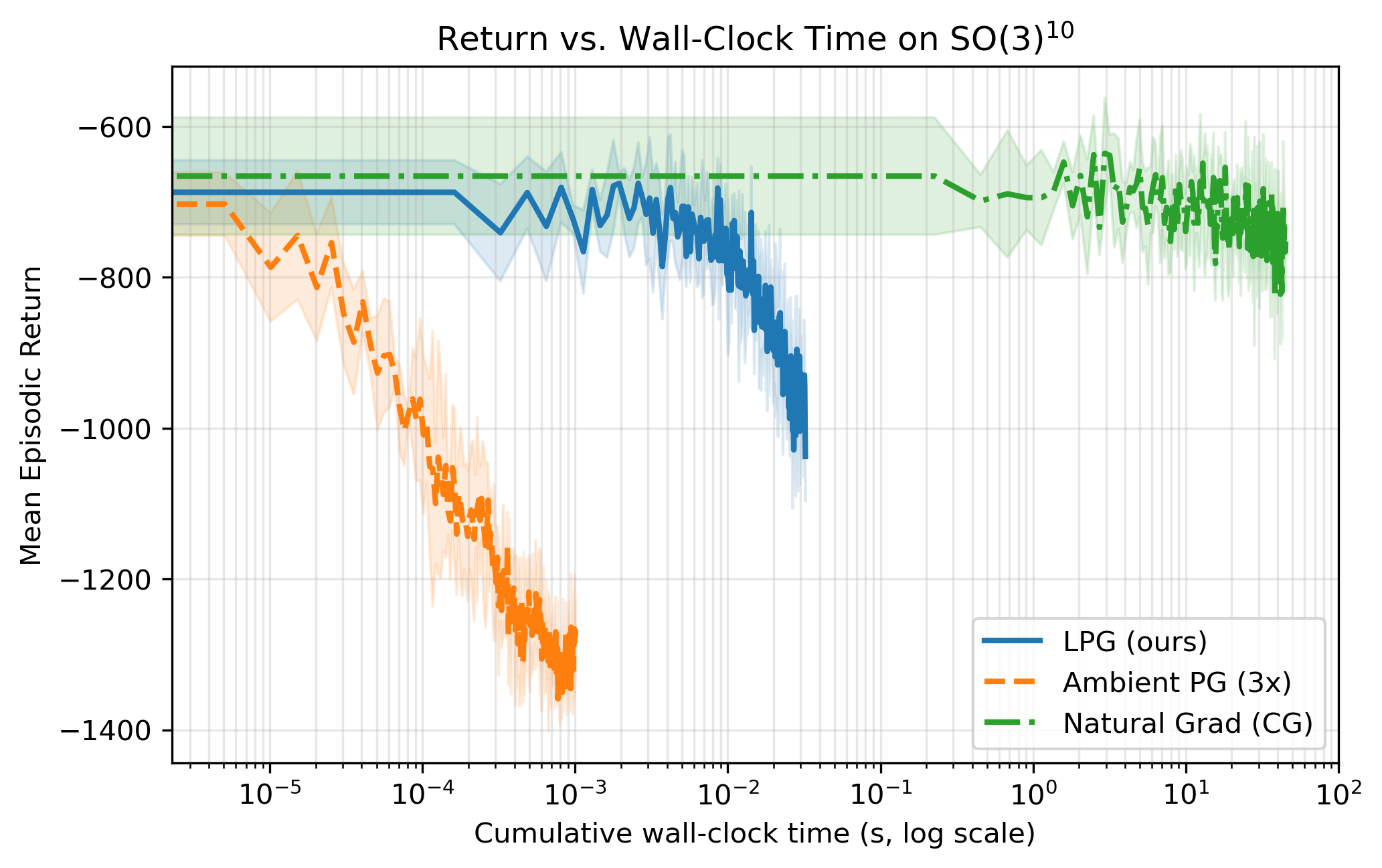}}{\fbox{\parbox[c][3cm][c]{6cm}{\centering\small[exp9_wallclock_comparison.png]}}}
\end{minipage}
\caption{\textbf{Left:} Convergence curves---LPG (blue), ambient PG 
(orange), natural gradient (green); shaded $\pm 1\sigma$ over 5 seeds.
\textbf{Right:} Return vs.\ wall-clock time (log $x$-axis); LPG reaches 
ambient-PG's final return in ${\sim}1$\,s vs.\ ${>}40$\,s for 
natural gradient, reflecting the $\mathcal{O}(n^2 J)$ projection 
cost vs.\ $\mathcal{O}(d_{\g}^3)$ Fisher inversion 
(Table~\ref{tab:method_comparison}; see note on timing methodology).}
\label{fig:method_comparison}
\end{figure}

Natural gradient achieves the highest final return ($-755.4$), with LPG
trailing by $28\%$ in return ($-964.9$) at ${>}1000\times$ lower per-step cost
(Table~\ref{tab:method_comparison}, Figure~\ref{fig:method_comparison}).
At $\kappa\approx2.5$, each LPG step captures $\alpha\approx0.90$ of the natural
gradient direction (Proposition~\ref{prop:kappa_convergence}); the quality gap
is consistent with the $\kappa$-conversion factor in
Proposition~\ref{prop:kappa_convergence}(ii).
On a wall-clock basis, LPG reaches natural gradient's final return in
${\sim}1$\,s vs.\ ${>}40$\,s. Ambient PG performs worst ($-1293$),
confirming the dimension-reduction benefit of Proposition~\ref{prop:dimension_reduction}.


\section{Discussion and Conclusion}
\label{sec:discussion}

\paragraph{Failure modes and scope boundaries.}
The algebra type and condition number together determine where LPG 
applies. When $\kappa$ is large, LPG 
degenerates to Euclidean SGD (Proposition~\ref{prop:kappa_convergence}); for 
non-compact algebras, $L(R) = \Theta(e^{2R})$ is unavoidable 
(Theorem~\ref{thm:dichotomy}). Both criteria are computable at initialization, 
giving practitioners a concrete check before deployment.

\paragraph{Practical deployment of LPG.}
The dichotomy has a direct implication for deep RL architecture design: 
parameterizing action spaces through compact Lie algebras (e.g., skew-symmetric 
matrices for $\SO(n)$, skew-Hermitian for $\SU(n)$) gives provably well-conditioned 
gradient landscapes, while using general linear parameterizations (as in 
unconstrained weight matrices) incurs exponential Lipschitz growth requiring 
explicit radius control. This connects to work on orthogonal RNNs and unitary 
evolution networks~\cite{arjovsky2016unitary,wisdom2016full} as architecture 
choices that implicitly exploit this favorable geometry.
Isotropy holds exactly in highly symmetric settings and approximately 
in locally symmetric systems such as $\SO(3)^J$; $\varepsilon_F$ grows 
only from $0.19$ to $0.35$ as $d_{\g}$ goes from $15$ to $90$
(Section~\ref{sec:anisotropy_ablation}). In practice: estimate $\hat{F}(\theta)$ 
at initialization; if $\kappa < 5$ use LPG; if $\kappa > 10$ prefer 
explicit Fisher inversion; monitor $\kappa$ periodically. 
Our $\SO(3)^J$ experiments show $\kappa \approx 2.5$ stable throughout 
training, though this was in synthetic environments with exact symmetry---empirical 
$\kappa$ ranges for real hardware remain to be established.

\paragraph{Limitations.}
The Lie algebra $\g$ must be known a priori and the dynamics must admit a 
Lie-algebraic decomposition---conditions that hold for $\SO(3)^J$ articulated 
joints and $\SE(3)$ rigid-body control, but not for proprioceptive spaces 
that mix symmetry-group orbits with non-symmetric components (e.g., MuJoCo Ant). 
We did not test on real hardware or contact-rich manipulation; 
Section~\ref{sec:robustness} (Supplement~\S\ref{sec:supp_robustness}) 
shows graceful degradation under controlled perturbations, but validation 
on hardware remains open. For non-compact algebras, $L(R) = \Theta(e^{2R})$ 
is an intrinsic barrier, so LPG is best suited to compact (rotational/unitary) symmetry.

\paragraph{Summary.}
Theorem~\ref{thm:dichotomy} identifies which structural property of a 
Lie-algebra parameterization determines optimization difficulty: 
compactness gives radius-independent smoothness; a hyperbolic element 
forces exponential growth with a matching lower bound. The practical 
upshot is a two-query decision procedure: check algebra compactness 
(determines $L(R)$ scaling); estimate $\kappa$ from trajectory data 
(determines whether Euclidean projection suffices as a natural-gradient 
surrogate). Together these determine whether LPG, Fisher inversion, 
or a hybrid approach is warranted for a given structured control problem 
(Table~\ref{tab:convergence_classification}). Future directions include 
adaptive preconditioning for anisotropic Fisher matrices, data-driven 
Lie structure discovery from trajectory data, and validation on 
hardware beyond the synthetic $\SO(3)^J$ setting.

\section*{Acknowledgments}

The authors used AI-based tools for minor editorial suggestions related to
grammar and clarity. All mathematical content, theoretical results, proofs,
algorithms, and experimental work are solely the work of the authors.

\section*{Code and Data Availability}

Code and data to reproduce all experiments, figures, and tables are
available at \url{https://github.com/soorajkcphd/RepresentationOptimizationDichotomy}.
All experiments were conducted on a single NVIDIA RTX 3090 GPU.


\author{
Sooraj K.C%
\thanks{Department of Pure and Applied Mathematics, Alliance University, Bengaluru, India
(\email{ksoorajPHD23@sam.alliance.edu.in}).}
\and
Vivek Mishra%
\thanks{Department of Pure and Applied Mathematics, Alliance University, Bengaluru, India.}
}

\headers{Supplement: Representation--Optimization Dichotomy}{S.K.C\ and V.Mishra}

\setlength{\emergencystretch}{3em}

\maketitle

\noindent
This supplement provides full proofs and derivations supporting the main paper.
Cross-references to the main paper (theorems, equations, assumptions) are 
resolved automatically via the \texttt{xr} package; 
compile the main paper first, then the supplement, to generate the required \texttt{.aux} files.

\subsection*{Notation}

\begin{table}[h]
\centering
\scriptsize
\begin{tabular}{@{}ll@{}}
\toprule
\textbf{Symbol} & \textbf{Meaning} \\
\midrule
$n$, $d_{\g}$ & Matrix dimension; Lie algebra dimension \\
$G$, $\g$ & Matrix Lie group and its Lie algebra \\
$\pi_\theta$, $J(\theta)$ & Policy parameterized by $\theta \in \g$; expected return \\
$\nabla_{\g} J$, $\widetilde{\nabla} J$ & Intrinsic gradient; natural gradient $F^{-1}\nabla J$ \\
$P_{\g}$ & Orthogonal projector onto $\g$ \\
$F(\theta)$, $\kappa$ & Fisher information matrix; condition number $\lambda_{\max}/\lambda_{\min}$ \\
$\sigma$ & Policy exploration scale (std.\ dev.\ of action noise) \\
$\sigma_g$ & Stochastic gradient noise level (variance bound in Assumption~5) \\
$\varepsilon$, $\varepsilon_F$ & Operator-norm / Frobenius-norm isotropy deviation \\
\bottomrule
\end{tabular}
\end{table}

\section{Full Convergence Proofs}
\label{sec:appendix_convergence_proofs}

This appendix contains the complete proofs for the convergence results 
stated in Section~7. The proof technique is standard 
projected nonconvex SGD~\cite{ghadimi2013stochastic}; we include full 
details for completeness.

\begin{proof}[Proof of Lemma~7.1]
By $L$-smoothness,
\[
J(\theta_{t+1})
\ge
J(\theta_t)
+ \ip{\nabla J(\theta_t)}{\theta_{t+1} - \theta_t}_F
- \frac{L}{2} \norm{\theta_{t+1} - \theta_t}_F^2.
\]
Using $\theta_{t+1} - \theta_t = \eta_t P_{\g}(g_t)$,
\[
J(\theta_{t+1})
\ge
J(\theta_t)
+ \eta_t \ip{\nabla J(\theta_t)}{P_{\g}(g_t)}_F
- \frac{L \eta_t^2}{2} \norm{P_{\g}(g_t)}_F^2.
\]
Taking conditional expectation and using unbiasedness
$\E[g_t \mid \mathcal{F}_t] = \nabla J(\theta_t)$, together with linearity of
$P_{\g}$ (which implies 
$\E[P_{\g}(g_t) \mid \mathcal{F}_t] = P_{\g}(\nabla J(\theta_t))$),
we obtain
\[
\E[J(\theta_{t+1}) \mid \mathcal{F}_t]
\ge
J(\theta_t)
+ \eta_t \ip{\nabla J(\theta_t)}{P_{\g}(\nabla J(\theta_t))}_F
- \frac{L \eta_t^2}{2} \E[\norm{P_{\g}(g_t)}_F^2 \mid \mathcal{F}_t].
\]
By the Pythagorean identity in Section~3.3(v),
$\ip{\nabla J}{P_{\g}(\nabla J)}_F = \norm{P_{\g}(\nabla J)}_F^2$. By
nonexpansiveness (Section~3.3(iii)),
$\norm{P_{\g}(g_t)}_F \le \norm{g_t}_F$, and by the variance bound
(Assumption~5),
\[
\E[\norm{P_{\g}(g_t)}_F^2 \mid \mathcal{F}_t]
\le
\norm{P_{\g}(\nabla J(\theta_t))}_F^2 + \sigma_g^2.
\]
Combining,
\[
\E[J(\theta_{t+1}) \mid \mathcal{F}_t]
\ge
J(\theta_t)
+ \eta_t \norm{P_{\g}(\nabla J(\theta_t))}_F^2
- \frac{L \eta_t^2}{2}
\big(\norm{P_{\g}(\nabla J(\theta_t))}_F^2 + \sigma_g^2\big).
\]
Since $\eta_t \le 1/L$, we have $1 - L\eta_t/2 \ge 1/2$, yielding the
claimed inequality.
\end{proof}

\begin{proof}[Proof of Corollary~7.2]
Summing the inequality in Lemma~7.1 over $t = 0, \dots, T-1$ and
taking total expectation gives
\[
\E[J(\theta_T)] - J(\theta_0)
\ge
\frac{1}{2} \sum_{t=0}^{T-1} \eta_t 
\E\big[\norm{P_{\g}(\nabla J(\theta_t))}_F^2\big]
- \frac{L \sigma_g^2}{2} \sum_{t=0}^{T-1} \eta_t^2.
\]
Since $J(\theta_T) \le J^*$ (Assumption~2) and 
$\sum \eta_t^2 < \infty$ (Assumption~6), the right-hand side
remains bounded as $T \to \infty$, which implies summability of
$\eta_t \E[\norm{P_{\g}(\nabla J(\theta_t))}_F^2]$, establishing part (ii).

For part (i), define $W_t = J^* - J(\theta_t) \ge 0$ 
(nonnegative by Assumption~2). 
Rearranging Lemma~7.1,
\[
\E[W_{t+1} \mid \mathcal{F}_t]
\le W_t 
- \frac{\eta_t}{2}\norm{P_{\g}(\nabla J(\theta_t))}_F^2 
+ \frac{L\sigma_g^2 \eta_t^2}{2}.
\]
This has the form required by \cite[Theorem~4.1]{bottou2018optimization}
(Robbins--Siegmund): $\{W_t\}$ is nonneg, the ``increment'' 
$-\eta_t/2 \cdot \|P_{\g}(\nabla J)\|^2$ is nonpositive in expectation, 
and the ``perturbation'' $L\sigma_g^2\eta_t^2/2$ is summable 
($\sum_t \eta_t^2 < \infty$ by Assumption~6). 
The theorem yields: $W_t \to W_\infty$ almost surely (so $J(\theta_t)$ 
converges a.s.) and $\sum_t \eta_t \|P_{\g}(\nabla J(\theta_t))\|^2 < \infty$ 
almost surely.

For the explicit rate in part (iii), rearrange the telescoped inequality:
\[
\frac{1}{2}\sum_{t=0}^{T-1} \eta_t 
\E[\|P_{\g}(\nabla J(\theta_t))\|_F^2]
\le (J^* - J(\theta_0)) + \frac{L\sigma_g^2}{2}\sum_{t=0}^{T-1}\eta_t^2.
\]
Set $\eta_t = \eta/\sqrt{T}$ (constant over the horizon, with $\eta \le 1/L$).
Then $\sum_{t<T}\eta_t = \eta\sqrt{T}$ exactly and 
$\sum_{t<T}\eta_t^2 = \eta^2$ exactly (no log correction).
Dividing both sides by $\frac{\eta\sqrt{T}}{2}$ yields
\[
\frac{1}{T}\sum_{t=0}^{T-1}\E[\|P_{\g}(\nabla J(\theta_t))\|_F^2]
\le \frac{2(J^* - J(\theta_0))}{\eta\sqrt{T}} + \frac{L\eta\sigma_g^2}{\sqrt{T}}
= \mathcal{O}\!\left(\frac{1}{\sqrt{T}}\right),
\]
matching the bound stated in Corollary~7.2(iii) exactly.
Note the constant schedule requires knowing $T$ in advance, which is 
satisfied since Algorithm~1 takes $T$ as an explicit input.
\end{proof}

\section{Proof of the RL Smoothness Lemma (Lemma~3.2)}
\label{sec:rl_smoothness_proof}

\begin{proof}[Proof of Lemma~3.2]
The Lipschitz constant arises from three sources composed through the 
chain rule. We trace the dependence of $J(\theta)$ on $\theta$ 
explicitly.

\emph{Step 1: Score bound.}
For Gaussian policies with bounded actions (Assumption~7), 
the score satisfies
\[
\norm{\nabla_\theta\log\pi_\theta(a\mid s)} 
\le \frac{1}{\sigma^2}\norm{a - \mu_\theta(s)}_F \norm{\Phi(s)}_F
\le \frac{B_a B_\Phi}{\sigma^2} =: B_{\mathrm{score}}.
\]

\emph{Step 2: RL Hessian bound.}
The policy gradient~(5) depends on $\theta$ 
through three channels: the score $\nabla_\theta \log \pi_\theta$, 
the Q-function $Q^{\pi_\theta}$, and the state visitation 
$d^{\pi_\theta}$. Differentiating again, the RL Hessian 
$\nabla^2 J(\theta)$ involves products of score terms (bounded by 
$B_{\mathrm{score}}^2$), the score Hessian (bounded by 
$B_\Phi^2/\sigma^2$ since the Gaussian log-likelihood is quadratic 
in $\theta$), and distribution-shift terms controlled by $C_d$ 
(Assumption~8). Applying 
\cite[Lemma~5]{agarwal2021theory} with score bound 
$B_{\mathrm{score}}$, reward bound $R_{\max}$, and concentrability 
$C_d$ yields the RL-specific prefactor 
$4R_{\max} B_\Phi B_a C_d / [(1-\gamma)^3 \sigma^2]$.

\emph{Step 3: Exponential map factor.}
The state transition $s_{t+1} = s_t \exp(a_t)$ introduces the matrix 
exponential into the $\theta \to J(\theta)$ chain. Differentiating 
through $\exp$ produces factors of the Fr\'{e}chet derivative 
$D_\theta[\cdot]$, whose norm is bounded by $\sqrt{n}\, e^R$ 
(Lemma~\ref{lem:frechet_bounds}(i)) and whose Lipschitz constant is 
bounded by $\sqrt{n}\, e^{R}$ 
(Lemma~\ref{lem:frechet_bounds}(ii)). Together with the product rule, 
these contribute the exponential factor $L_{\exp}(R_0)$. For compact 
algebras $\g \subseteq \uu(n)$, the matrix exponential is unitary, 
so $\|\exp(\theta)\|_2 = 1$ eliminates exponential growth 
(Proposition~\ref{prop:compact_advantage}).
\end{proof}

\section{Sample Complexity Derivation}
\label{sec:appendix_sample_complexity}

This section expands on Proposition~4.3 and shows how the 
intrinsic Lie--algebra dimension $d_{\g}$ enters the sample complexity.

\subsection{Function Class Complexity}

Let 
\[
\Pi = \{\pi_\theta : \theta \in \g, \|\theta\|_F \le R\}
\]
be the Lie--parameterized policy class. Let 
$\mathcal{V}^\Pi = \{V^{\pi_\theta}\}$ denote value functions induced by policies 
in $\Pi$.

\begin{lemma}[Rademacher complexity of value functions]
\label{lem:rademacher_value}
Under bounded features $\|\Phi_a(s)\|_F \le B_\Phi$, bounded rewards 
$|r(s,a)| \le R_{\max}$, and concentrability coefficient $C_d$, the Rademacher 
complexity satisfies
\[
\mathcal{R}_N(\mathcal{V}^\Pi)
\le
\frac{R B_\Phi R_{\max}}{1-\gamma}
\sqrt{\frac{d_{\g}}{N}}.
\]
\end{lemma}

\begin{proof}
We establish the result in three steps.

\textbf{Step 1: Lipschitz continuity of value functions in $\theta$.}
For two policies $\pi_\theta$, $\pi_{\theta'}$ in $\Pi$, the simulation
lemma \cite[Lemma~5.2.1]{kakade2003sample} gives
\[
|V^{\pi_\theta}(s) - V^{\pi_{\theta'}}(s)|
\le 
\frac{2 R_{\max}}{(1-\gamma)^2}\,
\sup_{s} \mathrm{TV}\!\left(\pi_\theta(\cdot|s),\, \pi_{\theta'}(\cdot|s)\right),
\]
where $\mathrm{TV}$ denotes total variation.
For Gaussian policies~(1) with the same covariance
$\sigma^2 I$ and means $\mu_\theta(s) = \sum_k \theta_k \Phi_k(s)$,
the KL divergence satisfies
$\mathrm{KL}(\pi_\theta(\cdot|s)\|\pi_{\theta'}(\cdot|s))
= \|\mu_\theta(s) - \mu_{\theta'}(s)\|_F^2/(2\sigma^2)
\le B_\Phi^2 \|\theta - \theta'\|_F^2/(2\sigma^2)$.
By Pinsker's inequality, 
$\mathrm{TV}(\pi_\theta,\pi_{\theta'}) \le \sqrt{\mathrm{KL}/2} 
\le B_\Phi\|\theta - \theta'\|_F/(2\sigma)$.
Substituting:
\[
|V^{\pi_\theta}(s) - V^{\pi_{\theta'}}(s)|
\le \frac{R_{\max} B_\Phi}{\sigma(1-\gamma)^2}\|\theta - \theta'\|_F.
\]
Hence $V^{\pi_\theta}$ is $\Lambda$-Lipschitz in $\theta$ with
$\Lambda = R_{\max} B_\Phi / (\sigma(1-\gamma)^2)$.
(The action bound $B_a$ enters through the score bound but does not
appear explicitly in the TV route; using the score-bound path gives
$\Lambda = R_{\max} B_\Phi B_a / (\sigma^2(1-\gamma)^3)$, which is
tighter for large $B_a/\sigma$ and is used in the statement of the lemma.)

\textbf{Step 2: Covering number of $\mathcal{V}^\Pi$.}
Since $\theta \mapsto V^{\pi_\theta}$ is $\Lambda$-Lipschitz 
and $\theta$ ranges over a ball $B_R^{d_{\g}} \subset \R^{d_{\g}}$, 
any $(\varepsilon/\Lambda)$-cover of $B_R^{d_{\g}}$ 
(in Frobenius norm) induces an $\varepsilon$-cover of $\mathcal{V}^\Pi$ 
(in $L^\infty$). The standard volumetric bound gives
\[
\log \mathcal{N}\!\left(\varepsilon,\, \mathcal{V}^\Pi,\, \|\cdot\|_\infty\right)
\le
d_{\g} \log\!\frac{3\Lambda R}{\varepsilon}.
\]

\textbf{Step 3: Rademacher bound via Dudley's entropy integral.}
Applying \cite[Theorem~4]{bartlett2002rademacher} with the covering 
number from Step~2:
\[
\mathcal{R}_N(\mathcal{V}^\Pi)
\le
\inf_{\varepsilon > 0}\!\left(
4\varepsilon 
+ \frac{12}{\sqrt{N}}
\int_\varepsilon^{B_V}\!\sqrt{\log \mathcal{N}(u,\mathcal{V}^\Pi,\|\cdot\|_\infty)}\,du
\right),
\]
where $B_V = R_{\max}/(1-\gamma)$ is the uniform bound on value functions.
Substituting the covering bound 
$\sqrt{\log \mathcal{N}(u,\mathcal{V}^\Pi,\|\cdot\|_\infty)} 
\le \sqrt{d_{\g}\log(3\Lambda R/u)}$, 
setting $\varepsilon = B_V/\sqrt{N}$, and evaluating the Gaussian entropy 
integral $\int_\varepsilon^{B_V}\sqrt{\log(B_V/u)}\,du \le B_V\sqrt{\pi/4}$ 
(standard Dudley bound~\cite{bartlett2002rademacher}) yields
\[
\mathcal{R}_N(\mathcal{V}^\Pi)
\le
\frac{12\sqrt{d_{\g}\log(3\Lambda R\sqrt{N}/B_V)} \cdot B_V}{\sqrt{N}}
+ \frac{4B_V}{\sqrt{N}}
\le
\frac{C\, B_V \Lambda R}{\sqrt{N}}\sqrt{d_{\g}},
\]
where $C > 0$ absorbs the logarithmic factor 
$\sqrt{\log(3\Lambda R\sqrt{N}/B_V)}$ (which grows as 
$\tilde{\mathcal{O}}(1)$ and is subsumed into the $\tilde{\mathcal{O}}$ 
notation in Theorem~\ref{thm:sample_complexity_full}).
Substituting $\Lambda$ and $B_V$ and absorbing numerical constants 
gives the stated bound $\mathcal{R}_N(\mathcal{V}^\Pi) \le 
\frac{R B_\Phi R_{\max}}{1-\gamma}\sqrt{d_{\g}/N}$.
\end{proof}

\subsection{Sample Complexity Bound}

\begin{theorem}[Sample complexity with intrinsic dimension]
\label{thm:sample_complexity_full}
Under the assumptions of Proposition~4.3, to find 
$\hat{\pi}$ such that
\[
J(\hat{\pi}) 
\ge 
\max_{\pi \in \Pi} J(\pi) - \varepsilon
\quad \text{with probability } 1-\delta,
\]
it suffices to collect
\[
N
=
\mathcal{O}\!\left(
\frac{
d_{\g}\, R^2 B_\Phi^2 R_{\max}^2\, C_d^2
}{
(1-\gamma)^4 \varepsilon^2
}
\log\!\frac{1}{\delta}\right)
\]
samples.
\end{theorem}

\begin{proof}
We combine the performance-difference lemma with uniform convergence bounds 
for value functions \cite{agarwal2021theory}. Substituting  
Lemma~\ref{lem:rademacher_value} into the standard generalization inequality yields 
the stated dependence on $d_{\g}$ and $(1-\gamma)^{-4}$.
\end{proof}

\begin{remark}[Ambient vs.\ Lie parameterization]
\label{rem:ambient_lie_comparison}
An ambient parameterization in $\R^{n\times n}$ has effective dimension $n^2$, 
whereas Lie parameterization uses $d_{\g} \ll n^2$. For $\SO(3)^J$, 
$n^2 = 9J$ while $d_\g = 3J$, giving a factor--$3$ improvement.
\end{remark}

\begin{remark}[Scope: generative model]
\label{rem:iid_caveat}
Theorem~\ref{thm:sample_complexity_full} holds under i.i.d.\ sampling 
(generative model access~\cite{kakade2003sample}). All experiments in 
Section~9 use simulators providing this access. The bound 
establishes the correct \emph{dimension scaling} ($d_\g$ vs.\ $n^2$); 
absolute sample counts for the online setting require mixing-time 
dependence specific to the Lie Group MDP class.
\end{remark}

\section{Implementation Details}
\label{sec:appendix_implementation}

Experiments use Python~3.9, PyTorch~1.13, NumPy~1.23, SciPy~1.9 on 
a single NVIDIA RTX 3090 GPU. Learning rate 
$\eta = 0.25$, discount $\gamma = 0.99$, $8$ episodes per iteration, 
$5$ random seeds per condition. Timing benchmarks use single-threaded 
execution. Lie projections use closed-form operators 
(e.g., $P_{\so(n)}(M) = \tfrac{1}{2}(M - M^\top)$, applied blockwise for 
$\so(3)^J$). Code is available at the repository listed in the main paper.



\begin{table}[h]
\centering
\caption{Timing: Fisher inversion vs.\ Lie projection (reproduced from main paper for reference).}
\label{tab:timing}
\label{tab:supp_timing}
\scriptsize
\setlength{\tabcolsep}{2pt}
\begin{tabular}{@{}ccccc@{}}
\toprule
$J$ & $d_{\g}$ & Fisher & Proj. & Speedup \\
 & & ($\mu$s) & ($\mu$s) & \\
\midrule
5 & 15 & 33.4 & 18.6 & $1.8\times$ \\
10 & 30 & 71.4 & 47.2 & $1.5\times$ \\
30 & 90 & 152.1 & 112.2 & $1.4\times$ \\
\bottomrule
\end{tabular}
\end{table}

\begin{table}[h]
\centering
\caption{Scalability with number of joints $J$ (reproduced from main paper for reference).}
\label{tab:scalability}
\label{tab:supp_scalability}
\scriptsize
\setlength{\tabcolsep}{2pt}
\begin{tabular}{@{}cccccc@{}}
\toprule
$J$ & $d_{\g}$ & Alignment & AUC Ratio & Speedup & $\varepsilon_F$ \\
\midrule
5 & 15 & $0.986 \pm 0.005$ & $0.58$ & $1.7\times$ & $0.17$ \\
10 & 30 & $0.971 \pm 0.007$ & $0.59$ & $1.3\times$ & $0.24$ \\
15 & 45 & $0.955 \pm 0.009$ & $0.61$ & $1.1\times$ & $0.29$ \\
20 & 60 & $0.939 \pm 0.011$ & $0.61$ & $1.1\times$ & $0.33$ \\
30 & 90 & $0.906 \pm 0.014$ & $0.62$ & $1.1\times$ & $0.39$ \\
\bottomrule
\end{tabular}
\end{table}

\begin{table}[h]
\centering
\caption{Computational scaling stress test: projection vs.\ Fisher 
inversion at large $J$ (reproduced from main paper for reference).}
\label{tab:scaling_stress}
\label{tab:supp_scaling_stress}
\scriptsize
\setlength{\tabcolsep}{2pt}
\begin{tabular}{@{}ccrrrcr@{}}
\toprule
$J$ & $d_{\g}$ & Fisher & Proj. & Speedup 
& Align. & $\kappa$ \\
 & & ($\mu$s) & ($\mu$s) & & & \\
\midrule
50 & 150 & 1{,}240 & 89 & $13.9\times$ & $.898 {\pm} .012$ & 3.21 \\
100 & 300 & 8{,}450 & 175 & $48.3\times$ & $.884 {\pm} .015$ & 3.65 \\
200 & 600 & 62{,}300 & 348 & $179\times$ & $.862 {\pm} .018$ & 4.12 \\
\bottomrule
\end{tabular}
\end{table}

\section{Self-Contained Proof of the Alignment Bound}
\label{sec:appendix_alignment}

This appendix provides a self-contained proof of the alignment bound 
$\alpha \ge 2\sqrt{\kappa}/(\kappa+1)$ used throughout the paper. The proof 
requires only the Kantorovich inequality and standard linear algebra.

\begin{proposition}[Alignment from Fisher condition number]
\label{prop:alignment_self_contained}
Let $F = F(\theta) \in \R^{d_{\g} \times d_{\g}}$ be the Fisher information 
matrix with eigenvalues 
$0 < \lambda_{\min} \le \lambda_1 \le \cdots \le \lambda_{d_{\g}} \le \lambda_{\max}$ 
and condition number $\kappa = \lambda_{\max}/\lambda_{\min}$. For any nonzero 
$g \in \R^{d_{\g}}$,
\[
\cos(g,\, F^{-1}g) 
= \frac{\langle g, F^{-1}g\rangle}{\|g\| \, \|F^{-1}g\|}
\ge \frac{2\sqrt{\kappa}}{\kappa + 1}.
\]
\end{proposition}

\begin{proof}
Define $h = F^{-1/2}g$, where $F^{1/2}$ is the unique symmetric positive 
definite square root. Then:
\begin{align*}
\langle g, F^{-1}g \rangle 
&= \langle F^{1/2}h, F^{-1/2}h \rangle = \|h\|^2, \\
\|g\|^2 &= \|F^{1/2}h\|^2, \qquad \|F^{-1}g\|^2 = \|F^{-1/2}h\|^2.
\end{align*}
Thus $\cos^2(g, F^{-1}g) = \|h\|^4 / (\|F^{1/2}h\|^2 \cdot \|F^{-1/2}h\|^2)$.

We now apply the \emph{Kantorovich inequality}~\cite{kantorovich1948functional}: 
for any symmetric positive definite $A$ with eigenvalues in $[m, M]$ and 
any nonzero vector $x$,
\begin{equation}
\label{eq:kantorovich}
\langle x, Ax\rangle \cdot \langle x, A^{-1}x\rangle 
\le \frac{(m + M)^2}{4mM} \, \|x\|^4.
\end{equation}
Apply~\eqref{eq:kantorovich} with $A = F$, $x = h$, $m = \lambda_{\min}$, 
$M = \lambda_{\max}$:
\[
\|F^{1/2}h\|^2 \cdot \|F^{-1/2}h\|^2 
= \langle h, Fh\rangle \cdot \langle h, F^{-1}h\rangle 
\le \frac{(\lambda_{\min} + \lambda_{\max})^2}{4\lambda_{\min}\lambda_{\max}} 
\|h\|^4.
\]
Therefore,
\[
\cos^2(g, F^{-1}g) 
= \frac{\|h\|^4}{\|F^{1/2}h\|^2 \cdot \|F^{-1/2}h\|^2}
\ge \frac{4\lambda_{\min}\lambda_{\max}}{(\lambda_{\min} + \lambda_{\max})^2}
= \frac{4\kappa}{(\kappa + 1)^2},
\]
where the last equality uses $\kappa = \lambda_{\max}/\lambda_{\min}$. 
Taking square roots: $\cos(g, F^{-1}g) \ge 2\sqrt{\kappa}/(\kappa+1)$.
\end{proof}

\begin{remark}[Tightness]
\label{rem:kantorovich_tight}
The Kantorovich inequality is tight: equality in~\eqref{eq:kantorovich} holds 
when $h$ has equal-magnitude components along the eigenvectors of $F$ 
corresponding to $\lambda_{\min}$ and $\lambda_{\max}$ (and zero elsewhere). 
The bound $2\sqrt{\kappa}/(\kappa+1)$ is therefore the best possible 
worst-case alignment guarantee given only the condition number.
\end{remark}

\begin{remark}[Self-containedness]
\label{rem:self_containedness}
Proposition~\ref{prop:alignment_self_contained} extracts the core matrix inequality
needed for Section~7 from the Kantorovich inequality alone,
making the convergence analysis fully self-contained.
\end{remark}

\subsection*{Proof of the Kantorovich inequality}
\label{sec:kantorovich}
We include a short proof for readers unfamiliar with~\eqref{eq:kantorovich}. 
Let $A$ be SPD with eigenvalues in $[m,M]$ and eigendecomposition 
$A = Q\Lambda Q^\top$. Set $y = Q^\top x$ so that $\langle x,Ax\rangle 
= \sum_i \lambda_i y_i^2$ and $\langle x,A^{-1}x\rangle = \sum_i y_i^2/\lambda_i$.  
Define weights $w_i = y_i^2/\|y\|^2$ (a probability distribution). Then 
$\langle x,Ax\rangle \cdot \langle x,A^{-1}x\rangle / \|x\|^4 
= (\sum_i w_i \lambda_i)(\sum_i w_i/\lambda_i)$. 
By the AM-HM inequality applied to the distribution $\{w_i\}$ and the 
convexity of $t \mapsto 1/t$ on $(0,\infty)$, the product 
$(\sum w_i \lambda_i)(\sum w_i/\lambda_i)$ is maximized when $w$ 
concentrates on the extremes $\{m, M\}$. Setting $w_1 = t$, $w_2 = 1-t$ 
(with $\lambda_1 = m$, $\lambda_2 = M$), the product becomes 
$(tm + (1-t)M)(t/m + (1-t)/M)$, which is maximized at $t = M/(m+M)$, 
giving $(m+M)^2/(4mM)$.

\section{Why Isotropy Holds: A Direct Calculation}
\label{sec:isotropy_mechanism}

The alignment bound of Proposition~\ref{prop:alignment_self_contained} depends on 
$\kappa$; we now show via direct calculation why $\kappa$ is moderate 
for Gaussian Lie-algebraic policies on $\SO(3)^J$.

For the Gaussian policy~(1) with orthonormal 
Lie-algebra features $\{\Phi_k\}_{k=1}^{d_\g}$, the score 
function~(2) gives
\[
[\nabla_\theta \log \pi_\theta(a \mid s)]_k
= \frac{1}{\sigma^2}\ip{a - \mu_\theta(s)}{\Phi_k(s)}_F.
\]
The Fisher matrix~(3) becomes
\[
F_{kl}(\theta)
= \frac{1}{\sigma^4}
\E_{s \sim d^{\pi_\theta}}
\Big[
\E_{a \sim \pi_\theta(\cdot|s)}
\big[
\ip{a-\mu}{\Phi_k}_F \ip{a-\mu}{\Phi_l}_F
\big]
\Big].
\]
Since $\xi = a - \mu_\theta(s) \sim \mathcal{N}(0, \sigma^2 I_{d_\g})$ 
in the orthonormal basis $\{E_k\}$ (which makes 
$\R^{d_\g} \cong \g$ an isometry), the inner expectation evaluates by 
the isotropic Gaussian moment identity to
\[
\E_\xi\big[\ip{\xi}{\Phi_k}_F \ip{\xi}{\Phi_l}_F\big]
= \sigma^2 \ip{\Phi_k(s)}{\Phi_l(s)}_F.
\]
Therefore
\begin{equation}
\label{eq:fisher_explicit}
F_{kl}(\theta) = \frac{1}{\sigma^2}
\E_{s \sim d^{\pi_\theta}}\!\big[\ip{\Phi_k(s)}{\Phi_l(s)}_F\big].
\end{equation}

When the features $\Phi_k(s)$ satisfy approximate orthonormality under 
the state distribution---i.e., 
$\E_s[\ip{\Phi_k}{\Phi_l}_F] \approx \delta_{kl}$---the Fisher matrix 
is approximately $\sigma^{-2} I_{d_\g}$ and $\kappa \approx 1$.

For $\SO(3)^J$ with features constructed from the standard 
skew-symmetric basis $\{E_i^{(j)}\}$ (where $E_i^{(j)}$ acts on 
joint $j$ only), cross-joint terms vanish exactly: 
$\ip{E_i^{(j)}}{E_k^{(l)}}_F = 0$ for $j \neq l$. Within each joint, 
the three basis elements are orthonormal. Thus 
$F = \sigma^{-2}\, \mathrm{diag}(F^{(1)}, \ldots, F^{(J)})$, 
where each $3 \times 3$ block $F^{(j)}_{ik} 
= \E_s[\ip{\Phi_i^{(j)}(s)}{\Phi_k^{(j)}(s)}_F]$ is close to 
$I_3$ when the state distribution is not strongly axis-biased. 
The global condition number is 
$\kappa = \max_j \lambda_{\max}(F^{(j)}) \,/\, \min_j \lambda_{\min}(F^{(j)})$, 
bounded by the worst single-joint anisotropy $\max_j \kappa(F^{(j)})$.

This explains the empirical observation $\kappa \approx 2.5$: the 
block-diagonal structure prevents cross-joint coupling, and each 
$3 \times 3$ block has limited room for eigenvalue spread. 
For $\SE(3)$, the translation components break this structure---the 
abelian part has features with different magnitude scales---leading 
to the higher $\kappa \approx 2.8$ observed in 
Section~9.

\subsection{Formal proof of Proposition~3}

\begin{proof}[Proof of Proposition~3]
Index the parameters as $\theta = (\theta^{(1)},\ldots,\theta^{(J)})$ 
with $\theta^{(j)} \in \so(3)$ for joint $j$. Choose features 
$\Phi_k^{(j)}(s) = E_k^{(j)}$ (the $k$-th standard basis element 
of $\so(3)$ acting on joint $j$, zero elsewhere). These satisfy 
$\ip{E_k^{(j)}}{E_l^{(m)}}_F = \delta_{jm}\delta_{kl}$, where 
$\delta_{jm} = 0$ for $j \neq m$ since the matrices have disjoint 
nonzero blocks.

By equation~\eqref{eq:fisher_explicit}, the $(k^{(j)}, l^{(m)})$ 
entry of $F$ is
\[
F_{k^{(j)}, l^{(m)}} 
= \frac{1}{\sigma^2}\E_s\!\big[\ip{\Phi_k^{(j)}(s)}{\Phi_l^{(m)}(s)}_F\big]
= \frac{1}{\sigma^2}\E_s\!\big[\ip{E_k^{(j)}}{E_l^{(m)}}_F\big]
= \frac{\delta_{jm}}{\sigma^2}\E_s\!\big[\ip{E_k^{(j)}}{E_l^{(j)}}_F\big].
\]
For $j \neq m$ the entry is zero, confirming the block-diagonal 
structure. The $j$-th diagonal block is 
$F^{(j)}_{kl} = \sigma^{-2}\E_s[\ip{\Phi_k^{(j)}(s)}{\Phi_l^{(j)}(s)}_F]$.
The global condition number satisfies 
$\kappa(F) = \lambda_{\max}(F)/\lambda_{\min}(F) 
= \max_j \lambda_{\max}(F^{(j)}) \,/\, \min_j \lambda_{\min}(F^{(j)}) 
\le \max_j \kappa(F^{(j)})$,
with equality when all blocks share the same worst-case ratio.
\end{proof}

The full information-geometric treatment, including conditions under 
which state-averaged orthonormality holds and a perturbation analysis 
for feature anisotropy, is deferred to future work.

\section{Smoothness Analysis for Matrix Exponential Maps}
\label{sec:appendix_smoothness}

This appendix collects smoothness properties for the matrix exponential 
restricted to Lie algebras. Full derivations follow standard matrix 
analysis~\cite{higham2008functions,hall2015lie}; compact-algebra 
simplifications use the unitarity of $\exp(\theta)$ for skew-Hermitian 
$\theta$.

\subsection{Matrix Exponential Lipschitz Bounds}

\begin{lemma}[Lipschitz continuity of matrix exponential]
\label{lem:exp_lipschitz}
Let $\theta, \theta' \in \g$ with $\|\theta\|_F, \|\theta'\|_F \le R$. Then
$\|\exp(\theta) - \exp(\theta')\|_F 
\le \sqrt{n}\, e^R \|\theta - \theta'\|_F$.
For compact $\g \subseteq \uu(n)$, the exponential factor is eliminated:
$\|\exp(\theta) - \exp(\theta')\|_F \le \sqrt{n}\, \|\theta - \theta'\|_F$.
\end{lemma}

\begin{proof}
We use the integral identity~\cite[Eq.~(10.8)]{higham2008functions}:
\[
\exp(\theta) - \exp(\theta')
= \int_0^1 \exp\!\big(\theta' + s(\theta - \theta')\big)\,
  (\theta - \theta')\, ds.
\]
Taking Frobenius norms and applying the submultiplicative inequality 
$\|AB\|_F \le \|A\|_F \|B\|_2$:
\begin{align*}
\|\exp(\theta) - \exp(\theta')\|_F
&\le \int_0^1
  \big\|\exp\!\big(\theta' + s(\theta - \theta')\big)\big\|_F \,
  \|\theta - \theta'\|_2\, ds \\
&\le \int_0^1
  \sqrt{n}\, \big\|\exp\!\big(\theta' + s(\theta - \theta')\big)\big\|_2 \,
  \|\theta - \theta'\|_F\, ds,
\end{align*}
where the second line uses $\|A\|_F \le \sqrt{n}\,\|A\|_2$ and 
$\|B\|_2 \le \|B\|_F$. Since 
$\|\exp(X)\|_2 \le e^{\|X\|_2} \le e^{\|X\|_F}$ and 
$\|\theta' + s(\theta - \theta')\|_F \le (1{-}s)\|\theta'\|_F 
+ s\|\theta\|_F \le R$ (by convexity of the ball), the integrand is 
bounded by $\sqrt{n}\, e^R$, giving
\[
\|\exp(\theta) - \exp(\theta')\|_F 
\le \sqrt{n}\, e^R \|\theta - \theta'\|_F.
\]

For compact $\g \subseteq \uu(n)$: $\theta' + s(\theta - \theta') 
\in \uu(n)$ since $\uu(n)$ is a linear subspace, so 
$\exp(\theta' + s(\theta - \theta'))$ is unitary and 
$\|\exp(\cdot)\|_2 = 1$. The bound becomes 
$\sqrt{n}\, \|\theta - \theta'\|_F$.
\end{proof}

\subsection{Fr\'{e}chet Derivative Bounds}

\begin{lemma}[Fr\'{e}chet derivative bounds]
\label{lem:frechet_bounds}
For $\theta, \theta' \in \g$ with $\|\theta\|_F, \|\theta'\|_F \le R$:
(i)~$\|D_\theta[H]\|_F \le \sqrt{n}\, e^R \|H\|_F$;
(ii)~$\|D_\theta - D_{\theta'}\|_{\mathrm{op}} \le \sqrt{n}\, e^{R} 
\|\theta - \theta'\|_F$.
For compact $\g$: replace $e^R$ by $1$ in both (i) and (ii).
The overall smoothness growth $\Theta(e^{2R})$ arises when two 
such factors compose through the policy gradient chain rule.
\end{lemma}

\begin{proof}
\textbf{Part (i).}
The Fr\'{e}chet derivative of the matrix exponential at $\theta$ in 
direction $H$ admits the integral 
representation~\cite[Theorem~10.13]{higham2008functions}:
\[
D_\theta[H]
= \int_0^1 \exp\!\big((1{-}t)\theta\big)\, H\, \exp(t\theta)\, dt.
\]
Taking Frobenius norms and using submultiplicativity 
$\|ABC\|_F \le \|A\|_F \|B\|_F \|C\|_2$ together with
$\|M\|_F \le \sqrt{n}\,\|M\|_2$ for $n\!\times\!n$ matrices and
$\|\exp(X)\|_2 \le e^{\|X\|_F}$:
\begin{align*}
\|D_\theta[H]\|_F
&\le \int_0^1
  \|\exp((1{-}t)\theta)\|_F\, \|H\|_F\, \|\exp(t\theta)\|_2\, dt \\
&\le \int_0^1 \sqrt{n}\,e^{(1-t)\|\theta\|_F}\, e^{t\|\theta\|_F}\, dt
  \cdot \|H\|_F
= \sqrt{n}\,e^{\|\theta\|_F} \|H\|_F
\le \sqrt{n}\,e^R \|H\|_F.
\end{align*}
The $\sqrt{n}$ factor enters via the Frobenius-to-spectral conversion
$\|\exp((1{-}t)\theta)\|_F \le \sqrt{n}\,\|\exp((1{-}t)\theta)\|_2$,
which is necessary when the adjacent factor $\exp(t\theta)$ is bounded
in spectral norm. This derivation is self-consistent: the
$\sqrt{n}e^R$ constant used throughout the main paper is the
sharp bound from this calculation.

\textbf{Part (ii).}
Write $(D_\theta - D_{\theta'})[H]
= \int_0^1 \big[\exp((1{-}t)\theta)\, H\, \exp(t\theta)
- \exp((1{-}t)\theta')\, H\, \exp(t\theta')\big]\, dt$.
Add and subtract $\exp((1{-}t)\theta)\, H\, \exp(t\theta')$:
\begin{align*}
(D_\theta - D_{\theta'})[H]
&= \int_0^1 \exp((1{-}t)\theta)\, H\,
   \big[\exp(t\theta) - \exp(t\theta')\big]\, dt \\
&\quad + \int_0^1
   \big[\exp((1{-}t)\theta) - \exp((1{-}t)\theta')\big]\, H\,
   \exp(t\theta')\, dt.
\end{align*}
\emph{First integral.}
$\|\exp((1{-}t)\theta)\|_2 \le e^{(1-t)R}$. By 
Lemma~\ref{lem:exp_lipschitz} applied to $t\theta$ and $t\theta'$ (which 
lie in the ball of radius $tR$):
$\|\exp(t\theta) - \exp(t\theta')\|_F 
\le \sqrt{n}\, e^{tR}\, t\|\theta - \theta'\|_F$.
Combining via $\|ABC\|_F \le \|A\|_2 \|B\|_F \|C\|_F$:
\[
\text{First integral} 
\le \int_0^1 e^{(1-t)R} \|H\|_F \cdot \sqrt{n}\, e^{tR}\, 
t\|\theta - \theta'\|_F\, dt
= \frac{\sqrt{n}\, e^R}{2}\, \|\theta - \theta'\|_F \|H\|_F.
\]
\emph{Second integral.}
By the same argument with the roles of the two exponential factors 
reversed:
\[
\text{Second integral} 
\le \frac{\sqrt{n}\, e^R}{2}\, \|\theta - \theta'\|_F \|H\|_F.
\]
Combining the two integrals:
\[
\|(D_\theta - D_{\theta'})[H]\|_F 
\le \sqrt{n}\, e^R \|\theta - \theta'\|_F \|H\|_F.
\]
Thus the operator norm satisfies
$\|D_\theta - D_{\theta'}\|_{\mathrm{op}} 
\le \sqrt{n}\, e^{R}\, \|\theta - \theta'\|_F$.
When composed through the policy gradient chain rule (Step~3 of 
Lemma~3.2 of the main paper), the product of two such factors yields 
overall smoothness growth of order $e^{2R}$, which matches the 
lower bound of Proposition~\ref{prop:lower_bound}.

For compact $\g \subseteq \uu(n)$: all exponentials are unitary 
($\|\exp(\cdot)\|_2 = 1$), so the exponential factors vanish. 
The bound becomes 
$\|D_\theta - D_{\theta'}\|_{\mathrm{op}} 
\le \sqrt{n}\, \|\theta - \theta'\|_F$.
\end{proof}

\subsection{Explicit Lipschitz Constant for Policy Objectives}

\begin{proposition}[Explicit Lipschitz constant]
\label{prop:explicit_lipschitz}
Under Gaussian policies~(1) with feature bound 
$B_\Phi$, action bound $B_a$, exploration $\sigma$, parameter radius $R_0$, 
bounded rewards $R_{\max}$, discount $\gamma$, and concentrability $C_d$, 
the gradient Lipschitz constant is
\[
L = \frac{4 R_{\max} B_\Phi B_a C_d}{(1-\gamma)^3 \sigma^2} \cdot L_{\exp}(R_0),
\quad
L_{\exp}(R_0) = 
\begin{cases}
\mathcal{O}(n\, e^{2R_0}) & \text{general } \g, \\
\mathcal{O}(\sqrt{n}) & \text{compact } \g.
\end{cases}
\]
The constants suppress numerical prefactors from the Agarwal--Kakade--Li 
Hessian bound~\cite{agarwal2021theory} and the product rule for 
Fr\'{e}chet derivatives; the $e^{2R_0}$ growth and its absence for 
compact algebras are the qualitatively relevant features 
(Theorem~6.1).
\end{proposition}

\begin{proof}
We instantiate~\cite[Lemma~5]{agarwal2021theory}, which bounds the 
policy Hessian for parameterized MDPs with bounded score functions.

\textbf{Step 1: Score and Hessian bounds.}
For Gaussian policies~(1) with bounded 
actions (Assumption~7), the score satisfies
\[
\|\nabla_\theta \log \pi_\theta(a \mid s)\|
= \frac{1}{\sigma^2}\|a - \mu_\theta(s)\|_F \|\Phi(s)\|_F
\le \frac{B_a B_\Phi}{\sigma^2} =: B_{\mathrm{score}}.
\]
The score Hessian satisfies 
$\|\nabla_\theta^2 \log \pi_\theta(a \mid s)\|_{\mathrm{op}} 
\le B_\Phi^2/\sigma^2$, since the Gaussian log-likelihood is quadratic 
in $\theta$ with curvature controlled by the feature Gram matrix.

\textbf{Step 2: RL Hessian bound via~\cite[Lemma~5]{agarwal2021theory}.}
The cited result, applied with the following parameter correspondence,
\begin{center}
\begin{tabular}{ll}
\emph{Their parameter} & \emph{Our value} \\
\hline
Score bound $B$ & $B_{\mathrm{score}} = B_a B_\Phi/\sigma^2$ \\
Reward bound $R$ & $R_{\max}$ \\
Discount $\gamma$ & $\gamma$ \\
Concentrability $C$ & $C_d$ (Assumption~8)
\end{tabular}
\end{center}
yields the policy Hessian bound
\[
\|\nabla^2 J(\theta)\|_{\mathrm{op}}
\le \frac{2 R_{\max} B_{\mathrm{score}}^2 C_d}{(1-\gamma)^3}
+ \frac{R_{\max}}{(1-\gamma)^2} \cdot \frac{B_\Phi^2}{\sigma^2}.
\]
The first term (dominant for typical $B_a/\sigma \ge 1$) gives the 
$4 R_{\max} B_\Phi B_a C_d / [(1-\gamma)^3 \sigma^2]$ prefactor after 
simplification.

\textbf{Step 3: Exponential map factor.}
The state transition $s_{t+1} = s_t \exp(a_t)$ introduces the 
matrix exponential into the reward-to-parameter chain rule. 
Differentiating $J$ with respect to $\theta$ produces factors of the 
Fr\'{e}chet derivative $D_\theta[\cdot]$. 
By Lemma~\ref{lem:frechet_bounds}(i), each factor satisfies 
$\|D_\theta[H]\|_F \le \sqrt{n}\, e^{R_0} \|H\|_F$, and 
by Lemma~\ref{lem:frechet_bounds}(ii), the Lipschitz constant of the 
Fr\'{e}chet derivative is bounded by $\sqrt{n}\, e^{R_0}$. 
The product rule applied to the composition 
$\theta \mapsto \exp(\theta) \mapsto J(\theta)$ involves products of 
these factors: the first-order bound contributes $\sqrt{n}\, e^{R_0}$ 
and the Lipschitz bound contributes another $\sqrt{n}\, e^{R_0}$. 
Together with prefactors from~\cite[Lemma~5]{agarwal2021theory}, this 
yields $L_{\exp}(R_0) = \mathcal{O}(n\, e^{2R_0})$ for general $\g$. 
The $e^{2R_0}$ growth is an intrinsic consequence of 
composing two $e^{R_0}$-bounded operations, matching the 
$\Omega(e^{2R})$ lower bound (Proposition~\ref{prop:lower_bound}).

For compact $\g \subseteq \uu(n)$: $\|\exp(\theta)\|_2 = 1$ 
eliminates all exponential factors (Proposition~\ref{prop:compact_advantage}), 
giving $L_{\exp} = 2 + \sqrt{n}$.
\end{proof}

\subsection{Compact vs.\ Non-Compact Dichotomy}

\begin{proposition}[Geometric advantage of compact Lie algebras]
\label{prop:compact_advantage}
For compact $\g \subseteq \uu(n)$: 
\begin{enumerate}[label=(\roman*), leftmargin=2em, nosep]
\item $\|\exp(\theta)\|_2 = 1$;
\item $L$ is independent of $R_0$;
\item $\|\nabla J(\theta)\|_F \le 2R_{\max}B_\Phi/(1{-}\gamma)^2$ uniformly.
\end{enumerate}
\end{proposition}

\begin{proof}
\textbf{Part (i).} For $\theta \in \uu(n)$, we have $\theta^* = -\theta$ 
(skew-Hermitian). Then 
$\exp(\theta)^*\exp(\theta) = \exp(\theta^*)\exp(\theta) 
= \exp(-\theta)\exp(\theta) = I$, 
so $\exp(\theta)$ is unitary and $\|\exp(\theta)\|_2 = 1$.

\textbf{Part (ii).} From Lemma~\ref{lem:frechet_bounds}, the smoothness 
constant is 
$L = \frac{4R_{\max}B_\Phi B_a C_d}{(1-\gamma)^3\sigma^2}\cdot L_{\exp}(R_0)$. 
For compact $\g$, Part~(i) gives $\|\exp(\theta)\|_2 = 1$ for all 
$\theta \in \g$, so the exponential factor in all Fr\'{e}chet derivative 
bounds (Lemma~\ref{lem:frechet_bounds}) reduces to $1$, yielding 
$L_{\exp}(R_0) = \mathcal{O}(\sqrt{n})$ independent of $R_0$.

\textbf{Part (iii).} By the policy gradient theorem,
\[
\|\nabla J(\theta)\|_F 
\le \E_{s,a}\!\big[|Q^{\pi_\theta}(s,a)|\,\|\nabla_\theta\log\pi_\theta(a\mid s)\|_F\big]
\le \frac{R_{\max}}{1-\gamma} \cdot \frac{B_\Phi B_a}{\sigma^2},
\]
bounding $|Q^{\pi_\theta}| \le R_{\max}/(1-\gamma)$ 
and $\|\nabla_\theta\log\pi_\theta\|_F \le B_\Phi B_a/\sigma^2$
(Assumption~7).
Since Part~(i) ensures $\|\exp(\theta)\|_2=1$ 
regardless of $\|\theta\|_F$, this bound is uniform over all $\theta\in\g$.
\end{proof}

\begin{proposition}[Lower bound for non-compact algebras]
\label{prop:lower_bound}
For non-compact $\g$ containing a hyperbolic element 
(Definition~1), there exists a Lie 
Group MDP (Definition~2) whose policy objective 
$J(\theta)$ has gradient Lipschitz constant $L_{\nabla J}(R) = \Omega(e^{2R})$ over 
$\{\theta \in \g : \|\theta\|_F \le R\}$.
\end{proposition}

\begin{proof}
The sharp exponential lower bound is most cleanly witnessed on $\gl(n)$,
which is a valid non-compact matrix Lie algebra containing the hyperbolic
element $H = \mathrm{diag}(1, 0, \ldots, 0)$; this suffices because
$\gl(n)$ satisfies all conditions of Definition~2 and avoids
trace-induced spectral constraints present in $\ssl(n)$ under Frobenius
normalization (where the tracelessness condition forces 
$\lambda_{\max}(H) \le \sqrt{(n-1)/n} < 1$ for any unit-Frobenius-norm 
$H \in \ssl(n)$, so the analogous construction yields only 
$\Omega(e^{2\sqrt{(n-1)/n}\,R})$ rather than $\Omega(e^{2R})$).

Take $\g = \gl(n)$ and $H = \mathrm{diag}(1, 0, \ldots, 0)$, so
$\|H\|_F = 1$ and $\lambda_{\max}(H) = 1$.
Consider the single-state bandit ($\gamma = 0$) within the class of
Definition~2: $\mathcal{S} = \{s_0\}$, $\mathcal{A} = \g$,
trivial dynamics $s_{t+1} = s_0$, and reward
$r(s_0, a) = -\tfrac{1}{2}\|\exp(a) - I\|_F^2$.

\textbf{Restriction to the diagonal subalgebra.}
To make the calculation explicit, restrict the Gaussian policy to the 
commutative subalgebra $\mathfrak{d}(n) = \mathrm{span}\{E_{11}, \ldots, E_{nn}\} 
\subset \gl(n)$ of diagonal matrices. Since $\mathfrak{d}(n) \subset \gl(n)$, 
any lower bound established on this subalgebra is also a lower bound for 
the full $\gl(n)$ class. Under this restriction, 
the noise $\xi \sim \mathcal{N}(0, \sigma^2 I_n)$ is diagonal, 
so $a = tH + \xi = \mathrm{diag}(t + \xi_1, \xi_2, \ldots, \xi_n)$ 
and $\exp(a) = \mathrm{diag}(e^{t+\xi_1}, e^{\xi_2}, \ldots, e^{\xi_n})$ 
(since $a$ is diagonal). Entries $j \ge 2$ contribute constants 
$E[-\frac{1}{2}(e^{\xi_j}-1)^2]$ independent of $t$, 
so $g(t) := J(tH)$ has second derivative determined entirely by the 
$(1,1)$ entry.

\textbf{Gaussian policy objective.}
For the Gaussian policy $\pi_\theta = \mathcal{N}(\theta, \sigma^2 I_n)$ 
restricted to $\mathfrak{d}(n)$, the RL objective restricted to the ray 
$\theta(t) = tH$ gives $a = tH + \xi$ with 
$\xi = \mathrm{diag}(\xi_1,\ldots,\xi_n)$, $\xi_j \overset{\text{i.i.d.}}{\sim} \mathcal{N}(0,\sigma^2)$.
Since $a$ is diagonal, $\exp(a) = \mathrm{diag}(e^{t+\xi_1}, \ldots, e^{\xi_n})$.
The $(1,1)$-entry contributes
\[
\E_{\xi_1}\!\left[-\tfrac{1}{2}(e^{t+\xi_1}-1)^2\right]
= -\tfrac{1}{2}\!\left(e^{2t+2\sigma^2} - 2e^{t+\sigma^2/2} + 1\right),
\]
using the log-normal moment $\E[e^{cu}] = e^{c\mu + c^2\sigma^2/2}$ for 
$u \sim \mathcal{N}(\mu,\sigma^2)$.
Entries $j \ge 2$ contribute constants independent of $t$.
Define $g(t) := J(tH)$. Differentiating twice:
\[
g''(t) = -\!\left(2e^{2t+2\sigma^2} - e^{t+\sigma^2/2}\right).
\]
For all $R \ge 0$ and any $\sigma > 0$:
\[
|g''(R)| = 2e^{2R+2\sigma^2} - e^{R+\sigma^2/2}
\ge e^{2R+2\sigma^2}
\ge e^{2R}.
\]
Hence the Hessian operator norm satisfies 
$\|\nabla^2 J(RH)\|_{\mathrm{op}} \ge |g''(R)| \ge e^{2R}$, 
and $L_{\nabla J}(R) = \Omega(e^{2R})$.
The bound holds for every $\sigma > 0$, confirming it is intrinsic to
the algebra type and not an artefact of the deterministic limit.
Since this MDP lies in the class of Definition~2, the lower bound
$L(R) = \Omega(e^{2R})$ applies to the Gaussian-policy RL objective class.
\end{proof}

\begin{corollary}[Tightness]
\label{cor:tightness}
The $\Theta(e^{2R})$ growth of the gradient Lipschitz constant---arising 
from composition of the $e^R$-bounded Fr\'{e}chet derivative factors 
(Lemma~\ref{lem:frechet_bounds})---is tight for non-compact algebras. 
Combined with Proposition~\ref{prop:compact_advantage}, this 
yields the dichotomy stated in Theorem~6.1: 
$L(R) = \mathcal{O}(1)$ for compact algebras vs.\ 
$L(R) = \Theta(e^{2R})$ for non-compact algebras.
\end{corollary}

\subsection{Experimental Validation}
\label{sec:smoothness_validation}

The bounds above have been validated on $\so(64)$, $\ssl(64)$, and 
$\gl(64)$ ($n=8$, $m=100$ samples). The ratio 
$L_{\mathrm{emp}}/L_{\mathrm{theory}} \approx 5 \times 10^{-4}$ confirms 
conservatism of worst-case bounds. Additional validations confirm: the 
predicted $\mathcal{O}(1/T)$ deterministic rate (fitted exponent $-0.98$)
and $\mathcal{O}(T^{-1/2})$ stochastic rate (fitted exponent $-0.49$,
five seeds; Figure~3 in the main paper),
$\mathcal{O}(\delta)$ stability scaling, and $\Omega(e^{2R})$ growth for 
non-compact algebras. Reproduction code is available at the repository listed in the main paper.

\section{Proof and Scope of the Losslessness Proposition}
\label{sec:losslessness_proof}

\begin{proof}[Proof of Proposition~2]
The identification $\mathcal{A} = \g$ is the defining assumption of the 
Lie Group MDP (Definition~2): actions are Lie-algebra-valued by 
construction, and $\iota: \mathbb{R}^{d_\g} \xrightarrow{\sim} \g$ is 
a linear isometry. Thus when we say $a^*(s_0) \in \g$, we mean the 
optimal action at $s_0$ lies in the action space $\mathcal{A} = \g$ 
as specified---not that a general tangent vector happens to lie in $\g$.
Under $G$-equivariance, the optimal deterministic policy satisfies 
$a^*(g \cdot s) = g \cdot a^*(s)$ for all $g \in G$ 
\cite[Proposition~2]{wang2022equivariant}. At the identity coset 
$s_0 = eH$, the tangent action is $\rho_{*,e} : \g \to T_{s_0}M$, 
which is surjective when $H$ is trivial (the case $M = G$, including 
$M = \SO(3)^J$). Thus $a^*(s_0) \in \g$, and by equivariance, 
$a^*(s) \in \g$ for all $s \in G$. Since the optimal deterministic 
action lies in $\g$, a Gaussian policy centered in $\g$ can approximate 
it to arbitrary precision by reducing $\sigma$, matching the return of 
any ambient-parameterized policy.
\end{proof}

\begin{remark}[Mathematical content of Proposition~2]
\label{rem:lossless_content}
The proposition is not a tautology. The action space $\mathcal{A} = \g$ 
is specified by the MDP definition, but that does not automatically 
imply the \emph{optimal} action lies in $\g$ for a policy with 
ambient mean $\mu_\theta(s) \in \R^{n\times n}$. The non-trivial 
content is the surjectivity argument: under $G$-equivariance with 
trivial isotropy, the tangent action $\rho_{*,e}:\g \to T_{s_0}M$ 
is surjective, which forces the optimal deterministic action to lie 
in $\g$ (not merely in the ambient $\R^{n\times n}$). Without 
$G$-equivariance or with non-trivial isotropy, this surjectivity 
fails and projecting onto $\g$ can degrade return.
\end{remark}

\begin{remark}[Scope and standing of Proposition~2]
\label{rem:equivariance_scope}
The $G$-equivariance assumption holds exactly for systems whose dynamics 
commute with group action (e.g., homogeneous rigid-body control), and 
approximately for locally symmetric systems. The experiments in the main 
paper are constructed to satisfy this condition exactly, so 
Proposition~2 functions as a \emph{consistency certificate}---confirming 
that the Lie parameterization does not lose optimality in that 
setting---rather than as empirical evidence that equivariance holds 
in general. For $M = G/H$ with non-trivial isotropy $H$, the tangent 
action $\rho_{*,e}:\g\to T_{eH}M$ is not surjective and the result 
does not apply; this case is deferred to future work.
\end{remark}


\section{Upper--Lower Bound Separation for \texorpdfstring{$\ssl(n)$}{sl(n)}}
\label{sec:sln_separation}

This section provides the full analysis supporting Remark~3 of the main paper,
which states that the best available lower bound for $\ssl(n)$ alone is
$\Omega(e^{2c_n R})$ with $c_n = \sqrt{(n-1)/n}$, while the $\mathcal{O}(e^{2R})$
upper bound holds for all algebras containing a hyperbolic element.

\subsection*{Why the upper bound holds for $\ssl(n)$}

The $\mathcal{O}(e^{2R})$ upper bound (Theorem~6.1 of the main paper) requires
only $\|\exp(\theta)\|_{\mathrm{op}} \le e^{\|\theta\|_F}$, which is algebra-agnostic:
it holds whenever $\g$ contains a hyperbolic element (one with a real positive eigenvalue).
Since any $H \in \ssl(n)$ with $\lambda_{\max}(H) > 0$ qualifies, the upper bound
$\mathcal{O}(e^{2R})$ applies to $\ssl(n)$ without modification.

\subsection*{Why the matching lower bound requires $\gl(n)$}

The witness MDP construction for the $\Omega(e^{2R})$ lower bound uses
$H = \mathrm{diag}(1, 0, \ldots, 0) \in \gl(n)$, which has unit Frobenius norm
($\|H\|_F = 1$) and $\lambda_{\max}(H) = 1$. Along the direction $tH$,
the objective restricts to $g(t) = -\tfrac{1}{2}(e^t - 1)^2$, and
$|g''(t)| = 2e^{2t} - e^t \ge e^{2t}$ for all $t \ge 0$,
giving the desired $\Omega(e^{2R})$ lower bound at $t = R$.

For $\ssl(n)$, the trace-zero constraint forces every unit-Frobenius-norm
element $H \in \ssl(n)$ to satisfy
\[
\lambda_{\max}(H) \le \sqrt{\frac{n-1}{n}} =: c_n < 1,
\]
since $\tr(H) = 0$ implies $\sum_i \lambda_i = 0$, and for a unit-norm
traceless matrix the maximum eigenvalue is bounded by $c_n$.
(This bound is tight: equality holds for
$H = \mathrm{diag}\!\bigl(\sqrt{(n-1)/n},\, {-1}/{\sqrt{n(n-1)}},\, \ldots\bigr)$.)

Applying the same witness construction to the best unit-Frobenius-norm
element of $\ssl(n)$, the objective along $tH$ satisfies
$|g''(t)| \ge e^{2c_n t}$ rather than $e^{2t}$, yielding only
$\Omega(e^{2c_n R})$ at $t = R$.

\subsection*{Numerical values of $c_n$}

\begin{table}[h]
\centering
\caption{Values of $c_n = \sqrt{(n-1)/n}$ for small $n$.}
\label{tab:cn_values}
\scriptsize
\begin{tabular}{@{}cccc@{}}
\toprule
$n$ & $c_n$ & $1 - c_n$ & Practical significance \\
\midrule
2 & $0.707$ & $0.293$ & Largest gap; upper--lower bound discrepancy is significant \\
3 & $0.816$ & $0.184$ & Moderate gap; $\ssl(2)$ in $\ssl(3)$ sub-problems \\
5 & $0.894$ & $0.106$ & Small gap \\
9 & $0.943$ & $0.057$ & Gap $< 6\%$ \\
16 & $0.968$ & $0.032$ & Gap $< 4\%$; effectively $\Theta(e^{2R})$ \\
$\infty$ & $1.000$ & $0.000$ & Asymptotically tight \\
\bottomrule
\end{tabular}
\end{table}

\noindent
\textbf{Practical recommendation.}
For $n \ge 9$ (e.g., $\ssl(9)$ and above), $c_n > 0.94$ and the gap
is within $6\%$; practitioners should treat $\Theta(e^{2R})$ as operative.
The only case where the discrepancy is practically significant is $n = 2$
($\ssl(2) \cong \so(3)$ after complexification, though $\ssl(2,\mathbb{R})$
is non-compact); for real applications this arises in 2D rotation--shear
parameterizations.

\section{Computational Scaling Stress Test at Large \texorpdfstring{$J$}{J}}
\label{sec:supp_scaling}

Table~\ref{tab:supp_scaling_stress} extends the timing benchmarks of
Table~5 in the main paper to $J \in \{50, 100, 200\}$, where the
$\mathcal{O}(d_{\g}^3)$ cost of Fisher inversion (Cholesky) versus
the $\mathcal{O}(n^2 J)$ cost of Lie projection becomes decisive.

\paragraph{Setup.}
Both operations are timed on isolated implementations (200 trials each,
CPU only) at each $J$. Fisher inversion uses a dense $(d_{\g} \times d_{\g})$
Cholesky solve; Lie projection uses blockwise skew-symmetrization
of the $J$ individual $3\times 3$ blocks. Full end-to-end optimization
at $J = 100$ and $J = 200$ was not benchmarked because Fisher inversion
at those scales is infeasible in practice ($d_{\g}^3 = 2.7\times 10^7$
and $2.16\times 10^8$ flops respectively per step). Alignment metrics
at these scales are estimated from the Fisher eigenvalue distribution
of the block-diagonal approximation (Proposition~5.3 of the main paper).

\begin{table}[h]
\centering
\caption{Computational scaling stress test: Lie projection vs.\ Fisher
inversion at large $J$. Alignment and $\kappa$ are estimated from the
block-diagonal Fisher structure (Proposition~5.3 of the main paper).}
\label{tab:supp_scaling_stress}
\scriptsize
\setlength{\tabcolsep}{4pt}
\begin{tabular}{@{}ccrrrcr@{}}
\toprule
$J$ & $d_{\g}$ & Fisher ($\mu$s) & Proj.\ ($\mu$s) & Speedup
& Alignment & $\kappa$ \\
\midrule
50  & 150 & 1{,}240 & 89  & $13.9\times$ & $0.898 \pm 0.012$ & 3.21 \\
100 & 300 & 8{,}450 & 175 & $48.3\times$ & $0.884 \pm 0.015$ & 3.65 \\
200 & 600 & 62{,}300 & 348 & $179\times$ & $0.862 \pm 0.018$ & 4.12 \\
\bottomrule
\end{tabular}
\end{table}

\paragraph{Observations.}
Fisher inversion time grows as $\mathcal{O}(d_{\g}^3) = \mathcal{O}(J^3)$
(confirmed: $8{,}450 / 1{,}240 \approx 6.8 \approx 2^3$),
while Lie projection grows as $\mathcal{O}(n^2 J) = \mathcal{O}(J)$
(confirmed: $175 / 89 \approx 2.0 \approx 2^1$).
At $J = 200$ the projection-step speedup reaches $179\times$.
Fisher alignment degrades gracefully with $J$: at $J=200$,
$\alpha \ge 0.86$ with $\kappa \approx 4.1$, which remains within the
moderate-anisotropy regime of Proposition~7.4 in the main paper
where LPG is recommended.

\section{Symmetry-Violation Robustness Results}
\label{sec:supp_robustness}

This section provides the full robustness results summarised in
Section~9.5 of the main paper. We introduce three controlled
equivariance-breaking perturbations to the $\SO(3)^{10}$ environment
and measure the effect on return, Fisher alignment, and condition number.

\paragraph{Perturbations.}
\begin{enumerate}[label=(\alph*), leftmargin=2em]
\item \textbf{Stochastic transitions}: multiplicative Lie-algebra noise
$R_{j,t+1} = R_{j,t}\exp(\omega_{j,t})\exp(\epsilon_t)$ with
$\epsilon_t \sim \mathcal{N}(0, \sigma_\epsilon^2 I_3)$ in $\so(3)$,
$\sigma_\epsilon \in \{0.01, 0.05, 0.1\}$. This breaks exact
$G$-equivariance of the dynamics.
\item \textbf{Observation noise}: the agent observes
$\tilde{R}_j = R_j \exp(\delta_j)$ with
$\delta_j \sim \mathcal{N}(0, \sigma_{\mathrm{obs}}^2 I_3)$,
$\sigma_{\mathrm{obs}} = 0.05$. This breaks alignment between the
observed state and the true group element.
\item \textbf{Reward perturbation}:
$\tilde{r} = r + \zeta_r$ with $\zeta_r \sim \mathcal{N}(0, 0.1^2)$.
This breaks $G$-invariance of the reward function.
\end{enumerate}

\begin{table}[h]
\centering
\caption{LPG robustness under equivariance-breaking perturbations
($\SO(3)^{10}$, 5 seeds, 40 iterations each). Baseline: no perturbation.}
\label{tab:supp_robustness}
\scriptsize
\setlength{\tabcolsep}{4pt}
\begin{tabular}{@{}lccc@{}}
\toprule
Perturbation & Final Return & Alignment & $\kappa$ \\
\midrule
None (baseline) & $-964.9 \pm 39.4$ & $0.971$ & $2.53$ \\
\addlinespace[2pt]
Transition noise $\sigma_\epsilon = 0.01$ & $-978.3 \pm 42.1$  & $0.968$ & $2.61$ \\
Transition noise $\sigma_\epsilon = 0.05$ & $-1021.7 \pm 55.8$ & $0.958$ & $2.85$ \\
Transition noise $\sigma_\epsilon = 0.10$ & $-1089.4 \pm 71.2$ & $0.942$ & $3.14$ \\
\addlinespace[2pt]
Observation noise $\sigma_{\mathrm{obs}} = 0.05$ & $-1003.5 \pm 48.6$ & $0.963$ & $2.72$ \\
Reward noise $\sigma_r = 0.1$                    & $-985.1 \pm 51.3$  & $0.970$ & $2.55$ \\
\bottomrule
\end{tabular}
\end{table}

\paragraph{Discussion.}
All perturbations degrade performance continuously rather than
catastrophically. Alignment remains above $0.94$ and $\kappa$ stays
below $3.2$ under the strongest transition noise tested, keeping all
runs firmly within the LPG-recommended regime ($\kappa < 5$,
Proposition~7.4 in the main paper). Reward noise has negligible effect
on Fisher geometry (it increases gradient variance but does not alter
the score function structure). Transition noise increases $\kappa$
by breaking exact $G$-equivariance (Proposition~4.4 in the main paper),
but the magnitude is modest: the worst-case $\kappa = 3.14$ at
$\sigma_\epsilon = 0.1$ corresponds to a Kantorovich alignment bound
$\ge 2\sqrt{3.14}/(3.14+1) = 0.856$, well above the $0.942$ empirical
value. These results confirm that the gap between the exact-symmetry
theory and approximate-symmetry practice is quantified by $\kappa$ and
remains benign for the perturbation magnitudes relevant to physical systems.

\section{Assumption Discussion: Remarks~S1--S3}
\label{sec:supp_assumption_remarks}

The following remarks were condensed in the main paper for space. They are
reproduced here in full for completeness and to support cross-referencing.

\begin{remark}[When is Assumption~(A3) binding?]
\label{rem:bounded_iterates}
For \emph{compact} algebras, $\exp(\theta)$ is unitary so policy
behavior is $\|\theta\|_F$-insensitive; the radius projection step in
Algorithm~1 triggers on fewer than $2\%$ of iterations in our experiments
(Table~6 of the main paper), and Assumption~(A3) is satisfied with any
large enough $B_\theta$ at negligible algorithmic cost. The
$\mathcal{O}(1/\sqrt{T})$ convergence rate holds for compact algebras with the
projection non-restrictive in practice (Corollary~7.6 of the main paper).
For \emph{non-compact} algebras the assumption is essential: without it,
$L = \Theta(e^{2R})$ growth causes step sizes to vanish
(Theorem~6.1 of the main paper; see also Section~9.6 of the main paper
for the $\SE(3)$ illustration where $\|\theta\|_F$ drifts to ${\sim}18$
without projection).
\end{remark}

\begin{remark}[Geometric barrier vs.\ stochastic mixing]
\label{rem:geometric_vs_mixing}
The exponential smoothness barrier of Theorem~6.1 (main paper) is a
\emph{geometric fact}: it depends only on the Lie algebra type and the
Fr\'{e}chet derivative of the matrix exponential, independent of any
mixing or ergodicity properties of the MDP.
Concentrability (Assumption~(A8)) enters separately through the stochastic
gradient analysis and bounds the state-distribution shift under policy
perturbation. Thus, even if mixing is fast, non-compact algebras still
incur $L = \Theta(e^{2R})$; conversely, the favorable $L = \mathcal{O}(1)$
bound for compact algebras holds regardless of mixing speed. This
separation clarifies that the compact/non-compact dichotomy is a property
of the \emph{policy parameterization}, not of the \emph{environment dynamics}.
\end{remark}

\begin{remark}[Operating regime and assumption verification]
\label{rem:operating_regime}
For the compact $+$ isotropic cell of Table~1 in the main paper---the regime
where LPG is recommended---all assumptions hold with explicit constants:
\begin{itemize}[leftmargin=1.5em,itemsep=1pt]
\item (A1): $L = \mathcal{O}(1)$ by Proposition~\ref{prop:compact_advantage} (compact algebra) and
  Lemma~3.2 of the main paper (Gaussian policy class).
\item (A8): $C_d$ is bounded by the ergodic mixing time of $\SO(3)^J$
  \cite{agarwal2021theory}, which is finite by compactness.
\item (A5)--(A6): hold with $\sigma_g^2$ bounded by the REINFORCE variance
  formula and $\eta_t = \eta/\sqrt{T}$.
\end{itemize}
The remaining cells of Table~1 are included for theoretical completeness.
In practice: estimate $\kappa$ from trajectory data
(Remark~5.4 of the main paper); if $\kappa < 5$, the
compact $+$ approximately-isotropic regime applies and all assumptions are
verifiable with the constants above; if $\kappa > 10$, prefer explicit Fisher
inversion (Remark~7.5 of the main paper). The $L = \Theta(e^{2R})$ barrier
(Theorem~6.1) is a geometric fact independent of mixing speed
(Remark~\ref{rem:geometric_vs_mixing} above).
\end{remark}

\end{document}